\documentclass[11pt]{amsart}
\usepackage{amsfonts}
\usepackage{amssymb}
\usepackage{slashbox}
\usepackage{hhline}
\def\a{\alpha}       \def\b{\beta}        \def\g{\gamma}
\def\d{\delta}           
\def\la{\lambda}           
\def\vr{\varrho}     \def\s{\sigma}       
\def\f{\varphi}      \def\p{\psi}         \def\z{\zeta}
\def\F{\Phi}         
\def\C{{\mathbb C}}  \def\D{{\mathbb D}}
\def\H{{\mathbb H}}  \def\N{{\mathbb N}}
\def\R{{\mathbb R}}

\def\cg{{\mathcal G}}    \def\ch{{\mathcal H}}

\def\({\left(}       \def\){\right)}
    
\newtheorem{proposition}{\sc Proposition}
\newtheorem{lemma}{\sc Lemma}
\newtheorem{theorem}{\sc Theorem}
\newtheorem{corollary}{\sc Corollary}
\newtheorem{ex}{\sc Example}

\begin{document}
%%%%%%%%%%%%%%%%%% title, etc. %%%%%%%%%%%%%%%%%%%%%%%%%%%%%%%%%%%%
\title[Maps intertwining two linear fractional maps]{Holomorphic
self-maps of the disk intertwining two linear fractional maps}
%%%%%%%%%%%%%%%%%%%%%%%%%%%%%%%%%%%%%%%%%%%%%%%%%%%%%%%%%%%%%%%%%%%
\author[M. D. Contreras]{Manuel D. Contreras}
\address{Departamento de Matem\'atica Aplicada II, Es\-cue\-la
T\'ec\-ni\-ca Su\-pe\-rior de In\-ge\-nie\-ros, Uni\-ver\-si\-dad
de Se\-vi\-lla, Ca\-mi\-no de los Des\-cu\-bri\-mien\-tos s/n,
41092 Seville, Spain} \email{contreras@us.es}
 \urladdr{http://personal.us.es/contreras}
%%%%%%%%%%%%%%%%%%%%%%%%%%%%%%%%%%%%%%%%%%%%%%%%%%%%%%%%%%%%%%%%%%%
\author[S. D\'{\i}az-Madrigal]{Santiago D\'{\i}az-Madrigal}
\address{Departamento de Matem\'atica Aplicada II, Es\-cue\-la
T\'ec\-ni\-ca Su\-pe\-rior de In\-ge\-nie\-ros, Uni\-ver\-si\-dad
de Se\-vi\-lla, Ca\-mi\-no de los Des\-cu\-bri\-mien\-tos  s/n,
41092 Seville, Spain} \email{madrigal@us.es}
%%%%%%%%%%%%%%%%%%%%%%%%%%%%%%%%%%%%%%%%%%%%%%%%%%%%%%%%%%%%%%%%%%%
\author[M. J. Mart\'{\i}n]{Mar\'{\i}a J. Mart\'{\i}n}
\address{Departamento de Matem\'aticas, M\'odulo 17, Edificio de Ciencias, Uni\-ver\-si\-dad Au\-t\'o\-no\-ma de Ma\-drid, 28049
Madrid, Spain}
\email{mariaj.martin@uam.es}
%%%%%%%%%%%%%%%%%%%%%%%%%%%%%%%%%%%%%%%%%%%%%%%%%%%%%%%%%%%%%%%%%%%
\author[D. Vukoti\'c]{Dragan Vukoti\'c}
\address{Departamento de Matem\'aticas, M\'odulo 17, Edificio de Ciencias, Uni\-ver\-si\-dad Au\-t\'o\-no\-ma de Ma\-drid, 28049
Madrid, Spain}
\email{dragan.vukotic@uam.es}
 \urladdr{http://www.uam.es/dragan.vukotic}
%%%%%%%%%%%%%%%%%%%%%%%%%%%%%%%%%%%%%%%%%%%%%%%%%%%%%%%%%%%%%%%%%%%
\subjclass[2000]{Primary: 30D05. Secondary: 37F10}
 \keywords{Intertwining, commutation, linear fractional
map, disk automorphism, angular derivative, iteration,
Denjoy-Wolff point}
\date{06 August, 2010}
\thanks{This work is supported by coordinated research grants
from MICINN, Spain: the first two authors are supported by
MTM2009-14694-C02-02 co-financed by FEDER (European Regional
Development Fund) and the remaining two authors by
MTM2009-14694-C02-01. All authors are also partially supported by the
Thematic Network (Red Tem\'atica) MTM2008-02829-E, Acciones
Especiales, MICINN. The first two and the fourth author are also partially sponsored by the ESF Network HCAA}
%%%%%%%%%%%%%%%%%%%%%%%%%%%%%%%%%%%%%%%%%%%%%%%%%%%%%%%%%%%%%%%%%%%
\begin{abstract}
We characterize (in almost all cases) the holomorphic self-maps of
the unit disk that intertwine two given linear fractional
self-maps of the disk. The proofs are based on iteration and a careful  analysis of the Denjoy-Wolff points. In particular, we characterize the maps that commute with a given linear fractional map (in the cases that are not already known) and, as an application, determine all ``roots'' of such maps in the sense of iteration (if any). This yields as a byproduct a short proof of a recent theorem on the embedding of a linear fractional transformation into a semigroup of holomorphic self-maps of the disk.
\end{abstract}
%%%%%%%%%%%%%%%%%%%%%%%%%%%%%%%%%%%%%%%%%%%%%%%%%%%%%%%%%%%%%%%%%%%
\maketitle
 \tableofcontents
%%%%%%%%%%%%%%%%%%%%%%%%% INTRODUCTION. %%%%%%%%%%%%%%%%%%%%%%%%%
\section*{Introduction}
%%%%%%%%%%%%%%%%%%%%%%%%%%%%%%%%%%%%%%%%%%%%%%%%%%%%%%%%%%%%%%%%%
\par
In what follows, writing $f\in\ch(\D,\D)$ or saying that $f$ is a
\textit{self-map\/} of the unit disk $\D$ will mean that $f$ is
analytic in $\D$ and $f(\D)\subset\D$. We will be particularly
interested in \textit{linear fractional transformations\/}:
$\f(z)=(a z+b)/(c z+d)$, $a d-b c\neq 0$, from now on, often
abbreviated as LFT's. The main purpose of this paper is to study
the self-maps $f$ of $\D$ that intertwine two prescribed linear
fractional self-maps $\f$ and $\p$ of the disk: $f\circ\f=\p\circ
f$.
\par
Our study is inspired by the special case when $\f=\p$; that is,
$f\circ\f=\f\circ f$. Related questions have a long history which
started with the pioneering works by Shields \cite{Shi} and Behan
\cite{Beh} on families of commuting self-maps of the disk and
continued with Cowen's articles \cite{C1, C2}. A considerable
amount of work has been done by Italian authors: see \cite{BG},
\cite{Br}, \cite{BTV}, or \cite{Vl}, to mention only a few papers
in the context of one complex variable. A relationship with the so-called polymorphic functions \cite{P1} should also be mentioned. Several partial results regarding the commutation with LFT's can be found in the recent paper \cite{JRS}.
\par
Two classical examples of intertwining relations are Schroeder's equation and Abel's equation from complex dynamics, where one of the intertwined maps is an LFT \cite{Va}, \cite{PC}, \cite{CDP1}. They are also important in the theory of composition operators \cite{CM, Sha, SSS} where they are used to analyze, for example, the compactness and the spectrum.
\par
Our aim is to show that in most situations where $\f$ and $\p$ are
both LFT's, the intertwining equation $f\circ\f=\p\circ f$ forces
$f$ to be an LFT as well. Such \textit{rigidity principles\/} are
frequent in the studies involving the use of different
generalizations of the Schwarz Lemma or its boundary versions
\cite{A}, \cite{BK}, \cite{BTV}. One of the key points in
answering the question on intertwining is precisely a careful
analysis of the Denjoy-Wolff points of the maps $\f$ and $\p$ and
the behavior of $f$ at these points.
\par
It should be noted, however, that there are quite a few exceptions to the rigidity principle in certain special situations. In order to make this clear, appropriate examples are provided wherever needed. We establish  that the intertwining equation can only hold for certain combinations of types of $\f$ and $\p$. Also, a conformality property of $f$ at the Denjoy-Wolff point of $\f$ is often relevant in this context. Even proving these initial properties seems to require a
considerable amount of work.
\par
We now indicate how the paper is organized and comment on some of our main results. In order to make the article as self-contained as possible, the essential background will be reviewed where needed but most preliminary facts are listed in Section~\ref{sec-backgr}.
\par
Following the standard classification of the non-identity self-maps of $\D$ into three cases: elliptic, parabolic, and hyperbolic, depending
on the location and properties of the Denjoy-Wolff points, in
Section~\ref{sec-compat-fcn-eqns} we analyze the compatibility of
dynamic types of $\f$ and $\p$ as a prerequisite for satisfying
the intertwining equation $f\circ\f=\p\circ f$. That is, given two linear fractional self-maps $\f$ and $\p$ of the disk of prescribed types, where $\f\neq id_\D$ and $\tau$ is the Denjoy-Wolff point of $\f$, the question is: does there exist a non-constant analytic function $f:\D\to\D$ conformal at $\tau$ and satisfying $ f\circ\f=\p\circ f$? The main results of Section~\ref{sec-compat-fcn-eqns} are summarized in the table below.
\par\smallskip
\begin{center}
Table 1
\par\smallskip
\noindent
\begin{tabular}[h]{||l l ||c|| c | c || c | c || c|c||}\hhline{|t:=========:t|}
\multicolumn{2}{||c||}{\backslashbox{$\varphi$}{$\psi$} } & $\textrm{id}_\D$ &
\multicolumn{2}{c||}{\begin{tabular}{l | l }\multicolumn{2}{c}{Elliptic} \\\hline Aut. & N-A
\end{tabular}} & \multicolumn{2}{c||}{\begin{tabular}{c | c }\multicolumn{2}{c}{Hyperbolic}
\\\hline Aut. & N-A\end{tabular}} & \multicolumn{2}{c||}{\begin{tabular}{c | c
}\multicolumn{2}{c}{Parabolic}
\\\hline Aut. & N-A\end{tabular}}\\
\hhline{|:=========:|}
Ell. & \hspace*{-0.29cm} \vline \ Aut. & $ \clubsuit$ & \hspace*{0.3cm}$ \clubsuit$  \hspace*{0.25cm} & No & No &  No  & No & No \\
\cline {2-9} & \hspace*{-0.29cm} \vline \ N-A & No &   No & $ \clubsuit$  & No & No & No & No \\
\hhline{|:=========:|}
Hyp.  & \hspace*{-0.29cm} \vline \ Aut. & No& No & $ \clubsuit$  & \hspace*{0.3cm}$ \clubsuit$  \hspace*{0.25cm}  & No & No & No\\
\cline {2-9} & \hspace*{-0.29cm} \vline \ N-A & No  & No & $ \clubsuit$  & $ \clubsuit$ & $ \clubsuit$  & No & No \\
\hhline{|:=========:|}
Par. & \hspace*{-0.29cm} \vline \ Aut. & $ \clubsuit$  & No & No & No & No & \hspace*{0.3cm}$ \clubsuit$  \hspace*{0.25cm}& No \\
\cline {2-9} & \hspace*{-0.29cm} \vline \ N-A & No  & No & No & No & No & No& $ \clubsuit$  \\
\hhline{|b:=========:b|}
\end{tabular}
\end{center}
\par\smallskip
The answer No in each case means that no such $f$ can exist regardless of the choice of $\f$ and $\p$ of the types required. The symbol $\clubsuit$ indicates that for \textit{some\/} LFT's $\f$ and $\p$ of the type specified one can find a function $f$ satisfying the conditions specified above. The abbreviations Aut and N-A mean automorphic and non-automorphic respectively.
\par
Once we know the cases when such an $f$ exists, it is natural to ask whether it is necessarily a linear fractional map. In Section~\ref{sec-fcn-eqns} we consider in detail all possible cases of intertwining and obtain several rigidity principles of this type. A separate and lengthy analysis in each different case is required in some theorems. The results are given at a glance in the following table.
\par\smallskip
\begin{center}
Table 2
\par\smallskip
\noindent
\begin{tabular}[h]{||l l ||c|| c | c || c | c || c|c||}\hhline{|t:=========:t|}
\multicolumn{2}{||c||}{\backslashbox{$\varphi$}{$\psi$} } & $\textrm{id}_\D$ &
\multicolumn{2}{c||}{\begin{tabular}{l | l }\multicolumn{2}{c}{Elliptic} \\\hline Aut. & N-A\end{tabular}} & \multicolumn{2}{c||}{\begin{tabular}{l | l }\multicolumn{2}{c}{Hyperbolic}
\\\hline Aut. & N-A\end{tabular}} & \multicolumn{2}{c||}{\begin{tabular}{l | l
}\multicolumn{2}{c}{Parabolic}
\\\hline Aut. & N-A\end{tabular}}\\
\hhline{|:=========:|}
Ell. & \hspace*{-0.3cm} \vline \ Aut. & No & \hspace*{0.13cm} No \hspace*{0.09cm}  &  &  &  &  &  \\
\cline {2-9} &\hspace*{-0.3cm} \vline \ N-A &  &    & Yes & & & & \\
\hhline{|:=========:|}
Hyp.  &\hspace*{-0.3cm} \vline \ Aut. &  &  & Yes & \hspace*{0.1cm} Yes \hspace*{0.08cm}&  &  &  \\
\cline {2-9} &\hspace*{-0.3cm} \vline \ N-A &    &  & Yes & Yes & Yes & &  \\
\hhline{|:=========:|}
Par. & \hspace*{-0.3cm} \vline \ Aut. & No &  &  &  &  & \hspace*{0.13cm} No \hspace*{0.091cm} & \\
\cline {2-9} &\hspace*{-0.3cm}  \vline \ N-A &   &  &  & & & & Yes \\
\hhline{|b:=========:b|}
\end{tabular}
\end{center}
\par\smallskip
(Naturally, blank spaces correspond to the cases ruled out by Table~1.) \par 
In Section~\ref{sec-applns} we give several applications of the
results obtained and the methods employed in
Section~\ref{sec-compat-fcn-eqns} and Section~\ref{sec-fcn-eqns}.
In particular, we solve completely the following problem: when does a self-map $f$ of the disk commute with a linear fractional self-map $\f$ of the disk? It should be noted that even in this special case
a direct application of earlier, seemingly more general, results
is not sufficient and additional discussions are required in order to deduce the desired conclusions.
\par
The paper ends with a new approach on what is sometimes called the
Koenigs embedding problem for semigroups of self-maps. The
machinery developed in the study of the roots allows us to deduce
in a quick way some results on embedding an LFT into such a
semigroup. This has been partly known for some time but in the full generality the statement has only been completed recently and by different methods \cite{BCD}.
\par\smallskip
\textbf{Acknowledgments}. We are most grateful to Professor Christian Pommerenke for some useful comments and for calling our attention to his paper \cite{P1}. The initial proof of Theorem~\ref{thm-inter-rigid} in this paper was based on the theory of models for discrete iteration in the unit disk, the key tool being the uniqueness of solutions of the well-known Schroeder and Abel equations for certain families of holomorphic self-maps of the unit disk. However, influenced by Pommerenke's proof of \cite[Theorem~3]{P1}, based only on the uniqueness for the Cauchy problem for ordinary differential equations, we have been able to avoid completely the approach based on models, thus simplifying the presentation.

%%%%%%%%%%%%%%%%% SOME BACKGROUND. %%%%%%%%%%%%%%%%%
\section{Some background}
 \label{sec-backgr}
%%%%%%%%%%%%%%%%%%%%%%%%%%%%%%%%%%%%%%%%%%%%%%%%%%%%%%%%%%%%%%%%

%%%%%%%%%%%%%%%%%%%%%%%
\subsection{Two simple criteria for an LFT to map the disk into itself}
\label{subs-lft-self}
%%%%%%%%%%%%%%%%%%%%%%%
It is clear that only some LFT's are self-maps of the disk.
However, it seems quite difficult to find a criterion explicitly
stated in the literature for deciding when an LFT given by
\begin{equation}
 \f(z)=\frac{az+b}{cz+d} \,, \qquad ad-bc\neq 0 \,,
 \label{lft}
\end{equation}
is a self-map of $\D$ in terms of $a$, $b$, $c$, and $d$. It is
our belief that such criteria deserve to be mentioned explicitly
since they are both simple and quite useful. We begin by stating
two simple tests of this kind. The first one is from the third author's paper \cite{M}.
\par
%%%%%%%%%%%%%%%%%%%%%
%% Proposition lft-self.
%%%%%%%%%%%%%%%%%%%%%
\begin{proposition} \label{prop-lft-self}
For a map $\f$ given by \eqref{lft} we have $\f(\D)\subset\D$ if
and only if
\begin{equation}
|b\overline d - a\overline{c}| + |ad-bc| \leq |d|^2-|c|^2 \,.
 \label{self-map}
\end{equation}
\end{proposition}
%%%%%%%%%%%%%%%%%%%%%
\par
Here is yet another criterion whose proof is simpler than
the one of Proposition~\ref{prop-lft-self} and which is equally
efficacious for our purpose. Note that, although the two criteria are clearly equivalent, it does not seem simple to deduce any of them from the other by elementary algebraic methods.
\par
%%%%%%%%%%%%%%%%%%%%%
%% Lemma lft-self.
%%%%%%%%%%%%%%%%%%%%%
\begin{lemma} \label{lem-lft-self}
An LFT given by \eqref{lft} maps $\D$ into itself if and only if
$|d|>|c|$ and
\begin{equation}
 2 |a\overline{b}-c\overline{d}|\le |c|^2+|d|^2-|a|^2-|b|^2 \,.
 \label{self-cond}
\end{equation}
Moreover, $\f$ is a disk automorphism if and only if
$|c|^2+|d|^2-|a|^2-|b|^2=a\overline{b}-c\overline{d}=0$.
\end{lemma}
%%%%%%%%%%%%%%
\begin{proof}
In order that an LFT $\f$ map $\D$ into itself, it must also map
the closed disk $\overline{\D}$ into itself. The condition
$|d|>|c|$ is clearly necessary in order for $\f$ to be analytic in
$\overline{\D}$. By the Maximum Modulus Principle, the condition
of being a self-map of $\overline{\D}$ is equivalent to
$|az+b|^2\le |cz+d|^2$ for all $z$ of modulus one, and this
happens if and only if
\[
 2\mathrm{Re\,}\{(a\overline{b}-c\overline{d}) z\} \le |c|^2+|d|^2-
 |a|^2-|b|^2 \,, \quad z\in\D \,.
\]
Taking the supremum over all numbers $z$ of modulus one, it is
immediate that the last condition is equivalent to
\eqref{self-cond}. The automorphism part is clear.
\end{proof}
%%%%%%%%%%%%%%%%%%%%%
\par
In the special case when $\f(0)=0$, that is, when $b=0$, both criteria  become simpler. In this form, they have been stated earlier in \cite{BFHS} and \cite[p.~203]{Sho}, for example.
%%%%%%%%%%%%%%%%%%%%%%%
\subsection{Angular limits and derivatives}
 \label{subs-basic-ang}
%%%%%%%%%%%%%%%%%%%%%%%
We will use $\angle$ before a limit to denote the angular
(non-tangential) limit. For a given self-map $f$ of $\D$ and a
point $p\in\partial\D$, if the angular limit $q=\angle\lim_{z\to
p} f(z)$ also belongs to $\partial\D$, then the angular limit
$\angle \lim_{z\to p} \dfrac{f(z)-q}{z-p}$ exists (on the Riemann
sphere $\widehat{\C}=\C\cup\{\infty\}$) and is different from zero
\cite[Proposition 4.13]{P2}. This limit is known as the
\textit{angular derivative\/} of $f$ at $p$. As is usual, we will denote it by $f^{\prime}(p)$.
\par
Closely related to these notions is the concept of (angular)
conformality at a point $p\in\partial\D$, which will play a major
role in our theorems and proofs. We recall that $f\in \ch(\D, \D)$
is said to be \textit{conformal\/} at $p\in\partial\D$ if the
angular limits
\[
q=\angle\lim_{z\to p}f(z)\in\overline{\D}\quad \text{and} \quad
\angle\lim_{z\to p}\frac{f(z)-q}{z-p}\neq 0,\infty
\]
exist (see \cite{P2}). Whenever $f$ is analytic at $p$, the
meaning of angular conformality coincides with the usual meaning:
$f^{\prime}(p)\neq 0$.

%%%%%%%%%%%%%%%%%%%%%%%
\subsection{Iteration and Denjoy-Wolff points}
 \label{subs-basic-iter}
%%%%%%%%%%%%%%%%%%%%%%%
Denote by $\N$ the set of all positive integers. As is usual, we
will write $f_n$ for the $n$-th iterate of a self-map $f$ of $\D$,
defined inductively by $f_1=f$ and $f_{n+1}=f\circ f_n$, $n\in\N$.
\par
It can easily be deduced from the Schwarz-Pick Lemma that a
non-identity self-map $f$ of the disk can have at most one fixed
point in $\D$. If such a unique fixed point in $\D$ exists, it is
usually called the \textit{Denjoy-Wolff point\/}. The iterates
$f_n$ of $f$ converge to it uniformly on the compact subsets of
$\D$ whenever $f$ is not a disk automorphism, but even for an
automorphism with a unique fixed point in $\D$ we will still refer
to such a point as the Denjoy-Wolff point of $f$.
\par
If $f$ has no fixed point in $\D$, the Denjoy-Wolff theorem
\cite{CM}, \cite{Sha} guarantees the existence of a unique point
$p$ on the unit circle $\partial\D$ which is the
\textit{attractive fixed point\/}, that is, the iterates $f_n$
converge to $p$ uniformly on the compact subsets of $\D$. Such $p$
is again called the \textit{Denjoy-Wolff point\/} of $f$. Note
that $f$ can have other (boundary) fixed points.
\par
Whenever $\f$ is a linear fractional self-map of the disk, its
Denjoy-Wolff point is a true fixed point since the map is
holomorphic in a disk larger than $\D$ and centered at the origin.
\par
When $p\in\partial\D$ is the Denjoy-Wolff point of $f$, then
$f^\prime(p)$ is actually real-valued and, moreover, $0<
f^\prime(p)\le 1$; see \cite{P2}. As is often done in the
literature, we classify the non-identity holomorphic self-maps of the disk into three categories according to their behavior near the Denjoy-Wolff point:
\begin{itemize}
 \item[(a)] \textit{elliptic\/} maps: the ones with a fixed point
inside the disk;
 \item[(b)] \textit{hyperbolic\/}: the ones with the Denjoy-Wolff
point $p\in\partial\D$ such that $f^\prime(p)<1$;
 \item[(c)] \textit{parabolic\/}: the ones with the Denjoy-Wolff
point $p\in\partial\D$ such that $f^\prime(p)=1$.
\end{itemize}
\par
%%%%%%%%%%%%%%%%%%%%%%%
\subsection{Standard simplifications by conjugation}
 \label{subs-simplif}
%%%%%%%%%%%%%%%%%%%%%%%
For each of the dynamic types, our study can be normalized by conjugation with an appropriate map so as to consider instead a simplified situation, either in the disk or in the half-plane.
\par
The standard disk automorphism $\f_p$ defined by
\begin{equation}
 \f_p(z)=\frac{p-z}{1-\overline{p} z} \,, \quad p\in\D
 \label{inv-autom}
\end{equation}
interchanges the points $p$ and $0$ and is its own inverse. It is easy to see that if $\f$ is an arbitrary elliptic automorphism of the disk different from the identity map and with Denjoy-Wolff point $p$, then
\[
 (\f_p\circ\f\circ\f_p)(z)=\la z \quad \text{and} \quad \la=
 \f^\prime(p)\in\partial\D\setminus\{1\}\,.
\]
This conjugation obviously reduces the study of elliptic
automorphisms to that of rotations.
\par
The following simple and standard procedure is suitable for both
the hyperbolic and parabolic maps. Let $\tau$ be the Denjoy-Wolff
point of a self-map $\f$ of $\D$ with $|\tau|=1$. The Cayley
transform $T_\tau(z)=\frac{\tau+z}{\tau-z}$ maps $\D$ conformally
onto the right half-plane $\H=\{z\,\colon\,\text{Re\,}z>0\}$ and
takes the point $\tau$ to infinity. Thus, to every self-map $f$ of
$\D$ there corresponds a unique self-map $g$ of $\H$, called the
\textit{conjugate map\/} of $f$, such that $g=T_\tau \circ f\circ
T_\tau^{-1}$ with the point at infinity as the Denjoy-Wolff point
(in $\H$).
\par
The method just described is particularly useful in the case of
linear fractional maps. It is not difficult to check that every
hyperbolic or parabolic linear fractional self-map $\f$ of $\D$
into itself is conjugate to a map of the form $\widetilde{\f}(w)=
A w+B$ with $A\ge 1$ and Re$B\ge 0$, with Denjoy-Wolff point at
infinity, and with $A=1/\f^\prime(\tau)$. Hence, $\f$ is parabolic
if and only if $A=1$ and hyperbolic if and only if $A>1$.

%%%%%%%%%%%%%%%%% CONFORMALITY AND COMPATIBILITY %%%%%%%%%%%%%%%%%
\section{Conformality of solutions and compatibility of dynamic
types}
 \label{sec-compat-fcn-eqns}
%%%%%%%%%%%%%%%%%%%%%%%%%%%%%%%%%%%%%%%%%%%%%%%%%%%%%%%%%%%%%%%%
In this (and the next) section we study the \textit{intertwining
equation\/}:
\begin{equation}
 f\circ\f=\p\circ f \,.
 \label{inter-eq}
\end{equation}
We first state and prove some basic necessary conditions for the
existence of solutions to \eqref{inter-eq}. They either tell us
that the solution must be conformal at the Denjoy-Wolff point of
$\f$ or indicate what dynamic types are required of $\f$ and $\p$
in order that the solution exist.
%%%%%%%%%%%%%%%%%%%%%%%
\subsection{Conformality of solutions}
 \label{subs-conform}
%%%%%%%%%%%%%%%%%%%%%%%
In what follows, $id_\D$ will denote the identity map acting on
the disk. Since a self-map of the unit disk other than $id_\D$ has
at most one fixed point in $\D$, if \eqref{inter-eq} holds and the
fixed point $p$ of $\f$ belongs to $\D$, then $f(p)$ is a fixed
point of $\p$ in $\D$.
\par
We begin by proving an auxiliary statement which may be of independent interest. It extends, and is modeled after, a similar but more special observation from \cite{Beh} on the fixed points of two intertwining linear fractional maps to the boundary. It also gives some information about the (angular) conformality of $f$ which seems to be a novelty in this context.
\par
%%%%%%%%%%%%%%%%%%%%%
%% Proposition prop-bdry.
%%%%%%%%%%%%%%%%%%%%%
\begin{proposition} \label{prop-bdry}
Suppose that \eqref{inter-eq} holds for the self-maps $f$, $\f$,
and $\p$ of $\,\D$, where $\,\f$ and $\,\p$ are LFT's. If the
Denjoy-Wolff point $p$ of $\f$ belongs to $\partial \D$ and $\p$
is not an elliptic automorphism, then $\angle \lim_{z\to
p}f(z)=q\in \overline{\D}$, where $q$ is the Denjoy-Wolff point of
$\,\p$. Moreover, if $\,q$ also belongs to $\partial \D$, then
there exists $\mu \in \widehat{\C}\setminus \{0\}$ such that
\[
\mu =\angle \lim_{z\to p}\frac{f(z)-q}{z-p}.
\]
In particular, $f$ is conformal at $p$ if and only if $\mu \neq
\infty $.
\end{proposition}
%%%%%%%%%%%%%%%%%%%%%%%
%%%%%%%%%%%%%%%%%%%%%%%
\begin{proof}
For each non-negative integer $n$, write $I_{n}=[1-2^{-n},
1-2^{-(n+1)})$. Define $\gamma\colon [0,1)\to\D$ by
\[
 \gamma(t)=\f_{n} \Big( ( 2^{n+1}t-(2^{n+1}-2)) \f(0) \Big)
 \,, \quad t\in I_{n} \,.
\]
Clearly, $\g$ is continuous in $(1-2^{-n},1-2^{-(n+1)})$ for all
$n$, so it is only left to check its continuity at each point of
the form $1-2^{-n}$. It is obvious that
\[
 \lim_{t\searrow 1-2^{-n}}\gamma(t)=\f_n(0)
\]
and also
\begin{eqnarray*}
 \lim_{t\nearrow 1-2^{-n}} \gamma(t) &=& \lim_{t\nearrow 1-2^{-n}}
 \f_{n-1}\Big( ( 2^{n}t-(2^{n}-2) ) \f(0)\Big)
\\
  &=&\f_{n-1}\left( \f(0)\right)=\f_n(0) \,,
\end{eqnarray*}
hence $\gamma$ is continuous in $[0,1)$ and so is a curve in the
unit disk.
\par
The segment $S=[0,\f(0)]$ is a compact subset of $\D$. On the one
hand, by the Denjoy-Wolff Theorem, the sequence
$\{\f_{n}\}_{n=1}^\infty$ converges to $p$ uniformly on $S$ and
therefore $\lim_{t\to1}\gamma(t)=p$. On the other hand, it follows
again from the Denjoy-Wolff Theorem that the sequence
$\{\p_{n}\}_{n=1}^\infty$ converges to $q$ uniformly on $f(S)$.
\par
By an inductive argument, the intertwining equation
\eqref{inter-eq} easily implies that $f\circ\f_n=\p_n\circ f$.
Thus, given $t\in I_{n}$ we have $f(\gamma(t))=f(\f_{n}(w))=
\p_{n}(f(w))$ for some point $w=w(t)\in [0,\f(0)]$. Therefore, we
conclude that $\lim_{t\to1}f(\gamma(t))=q$. Finally, by
Lindel\"{o}f's classical theorem \cite[\S 12.10]{R} it follows
that $\angle\lim_{z\to p}f(z)=q$.
\par
Assume now that $q$ also belongs to $
\partial \D$ and consider $g(z)=p\overline{q}f(z)$, $z\in\D$.
Clearly, $g$ is a holomorphic self-map of $\D$ which has $\,p$ as
a fixed point, and the corresponding angular limit exists.
Therefore, there exists $\delta\in (0,+\infty )\cup \{+\infty\}$
such that
\[
 \delta =\angle \lim_{z\to p}\frac{g(z)-p}{z-p}.
\]
The existence of the number $\mu$ defined in the statement follows
immediately by taking $\mu=\overline{p}\,q\,\d$.
\end{proof}
%%%%%%%%%%%%%%%%%%%%%%%
\par
There exists a function $f\in\ch(\D,\D)$ which is not conformal at
$p$ but still satisfies the intertwining equation $f\circ\f=
\p\circ f$, with either both $\f$ and $\p$ elliptic or both
hyperbolic, as will be seen from our later
Example~\ref{ex-non-univ} and Example~\ref{ex-noncom-infder}.
However, this is impossible when $\f$ and $\p$ are parabolic, as
the following result shows.
\par\smallskip
\begin{theorem}
\label{thm-comp} Suppose that \eqref{inter-eq} holds for the
self-maps $f$, $\f$, and $\p$ of $ \D$, where \ $\f$ and $\p$ are
parabolic LFT's. Likewise, let $p$ and $q$ be the Denjoy-Wolff
points of $\f$ and $\p$, respectively. Then, $f$ is conformal at
$p$ and actually
\[
 p\,\overline{q}\,f^{\prime}(p)\in (0,+\infty ) \quad and
 \quad f^{\prime}(p)=\frac{\f^{\prime\prime}(p)}{\p^{\prime\prime}
 (q)}\,.
\]
\end{theorem}
%%%%%%%%%%%%%%%%%%%%%%%
\begin{proof}
By Proposition~\ref{prop-bdry}, in order to prove that $f$ is conformal
at $p$, we need only check that $\mu\neq\infty$. We apply the
procedure described in Subsection~\ref{subs-simplif}. Consider the function $g:\H\to \H$ defined as $g=T_{q}\circ f\circ T_{p}^{-1}$, where $T_{p}$ and $T_{q}$ are the associated Cayley maps with respect to $p$ and $q$. Since $g$ is a holomorphic self-map of $\H$, we know that there exists $c\ge 0$ such that $c=\angle \lim_{w\to \infty } \dfrac{g(w)}{w}$. Moreover, by the very definition of $\mu $, we have that $c\,p\,\overline{q}=1/\mu$. So, it is only left to see that $c>0$.
\par
By transferring the intertwining equation from $\D$ to $\H$, we
find that $g\circ \widehat{\f }=\widehat{\p }\circ g$, where
\[
 \widehat{\f }(w)=\(T _{p}\circ \f \circ T_{p}^{-1}\)(w)=w+a
\]
and
\[
 \widehat{\p }(w)=\(T _{q}\circ \p \circ T _{q}^{-1}\)(w)=w+b\,,
\]
for some non-zero complex numbers $a$ and $b$ such that
$\text{Re}\, a,\text{Re}\, b\ge 0$  and for all $w\in \H$.
Iterating, we also obtain that, for every $w\in \H$ and for all
$n\in \mathbb{N}$ we have
\[
 g(w+na)=g(w) + n b\,.
\]
From here we get that
\[
 \frac{g(n+n a)}{n+n a} (1+a)= \frac{g(n)}{n}+b
\]
and since the sequences $\{n\}_{n=1}^\infty$ and
$\{n+na\}_{n=1}^\infty$ tend non-tangentially to infinity, by
letting $n\to\infty$ we deduce that $c (1+a)=c+b$, hence $ca=b$.
Since $b\neq 0$, we conclude that $c>0$ as desired.
\par
A tedious but straightforward computation shows that $a=p
\f^{\prime\prime}(p)$ and $b=q \p^{\prime\prime}(q)$. Thus,
recalling that $c\,a=b$, we finally have that
\[
 f^{\prime}(p)=\mu=\frac{1}{c}\frac{q}{p}=\frac{p\f^{\prime\prime}
 (p)}{q\p^{\prime\prime}(q)}\frac{q}{p}=\frac{\f^{\prime\prime}
 (p)}{\p^{\prime\prime}(q)}\,.
\]
\end{proof}
%%%%%%%%%%%%%%%%%%%%%

%%%%%%%%%%%%%%%%%%%%%%%
\subsection{Compatibility of dynamic types of the intertwining LFT's}
 \label{subs-compat-types}
%%%%%%%%%%%%%%%%%%%%%%%
Our next two theorems tell us that assuming that one of the maps
$\f$ or $\p$ is of certain dynamic type forces the other to be of
certain type (often the same) in order that the solution to
\eqref{inter-eq} exist.
\par
We will frequently use the term \textit{rational elliptic
automorphism\/} for an elliptic automorphism $\f$ conjugate to the
map $R_\la(z)=\la\,z$, where $\la^n=1$ for some positive integer
$n$. This is clearly equivalent to $\f$ being idempotent:
$\f_n=id_\D$ for some $n\in\N$. However, the term ``rational
elliptic automorphism'' is very common in dynamics and we will use
it here as well.
%%%%%%%%%%%%%%%%%%%%%
%% Theorem comp-1.
%%%%%%%%%%%%%%%%%%%%%
\begin{theorem}
\label{thm-comp-1} Let $\f$ and $\p$ be linear fractional
self-maps of $\D$ and let $f\in \ch(\D,\D)$ be a non-constant
function. Assume that the intertwining equation \eqref{inter-eq}
holds. Then we have the following conclusions:
\begin{enumerate}
 \item[(a)]
If $\f=id_\D$, then $\p=id_\D$.
 \item[(b)]
If $\p=id_\D$ and $\f\neq id_\D$, $p$ is the Denjoy-Wolff point of
$\f$ and there exists the angular limit $\angle\lim_{z\to p}
f(z)$, then $\f$ is either a rational elliptic automorphism or a
parabolic automorphism.
\end{enumerate}
\end{theorem}
%%%%%%%%%%%%%%%%%%%%%%%
\par
%%%%%%%%%%%%%%%%%%%%%%%
\begin{proof}
(a) This is the easy case because \eqref{inter-eq} simply reads
$f=\p\circ f$. Since $f$ is not a constant function, $f(\D)$ is an
open set and for all $w$ in this set we have $w=\p(w)$, which
shows that $\p=id_\D$.
\par\smallskip
(b) Let $\f$ be elliptic; then $p\in\D$. Suppose $\f$ is not an
automorphism. On the one hand, just like earlier, \eqref{inter-eq}
implies that $f\circ\f_n=f$ for all $n\in\N$. On the other hand,
$\f_n\to p\,$ pointwise in $\D$, hence $f\equiv f(p)$ in $\D$.
Since $f\not \equiv const$, it follows that $\f$ must be an
automorphism.
\par
We will now show that $\f$ is rational. Write $\la=\f^{\prime}%
(p)\in\partial\D$. Set $g=\f_{q}\circ f\circ\f_{p.}$ Notice that
$g(0)=0$ and $g(\la z)=g(z)$ for all $z\in\D$. Moreover, let
$\sum_{n=1}^{\infty}a_{n}z^{n}$ be the Taylor series for $g$
around zero. Then
\[
a_{n}\la^{n}=a_{n},\qquad\text{for all }n\in\mathbb{N}.
\]
Since $f$ is not constant, the function $g$ is not identically
zero and we conclude that there exists $n\in\mathbb{N}$ such that
$\la^{n}=1$. Hence $\f$ is rational.
\par
Now suppose that $\f$ is not elliptic and let us show that it must
be a parabolic automorphism. We can transfer the equation
$f\circ\f=f$ to the right half-plane by composing with the Cayley
map associated with $p$ and its inverse as in the proof of
Theorem~\ref{thm-comp}. This leads to the following equation for
the corresponding self-maps $\widehat{f}$, $\widehat{\f}$ of the
right half-plane: $\widehat{f}\circ\widehat{\f}=\widehat{f}$,
where $\widehat{\f}(w) =a w+b$ with $a\ge 1$ and Re\,$b\ge 0$. In
the case $a>1$ it follows by iteration that
\[
 \widehat{f}\(a^n w+\dfrac{a^n-1}{a-1} b\)=\widehat{f}(w) \,,
 \quad w\in\H\,.
\]
Since $\lim_{n\to\infty} a^n w+\dfrac{a^n-1}{a-1} b =\infty$ non-tangentially and the angular limit $\angle \lim_{z\to p} f(z)$ exists (and hence so does $\angle \lim_{w\to \infty} \widehat{f}(w)$), it follows that $f\equiv const$.
\par
Thus, in order that a non-trivial solution to \eqref{inter-eq}
exist, we must have $a=1$, that is, $\f$ must be parabolic. If
Re\,$b>0$ then $w+n b\to\infty$ non-tangentially. Again, by
\eqref{inter-eq} we have that $\widehat{f}(w+n b)=\widehat{f} (w)$
for all $w\in\H$ and all $n\in\N$ and once more it follows that
$f$ is constant. Therefore Re\,$b=0$, so $\f$ is a parabolic
automorphism.
\end{proof}
%%%%%%%%%%%%%%%%%%%%%%%
\par\smallskip
Both situations referred to in part (b) of
Theorem~\ref{thm-comp-1} are actually possible.

%%%%%%%%%%%%%%%%%%%%%
%% Example comp-pa.
%%%%%%%%%%%%%%%%%%%%%
\begin{ex} \label{ex-comp-pa}
Let $\f(z)=-z$ and $f(z)=z^2$. Then $\f$ is a rational elliptic
automorphism with $n=2$ and $f=f\circ\f$.
\end{ex}
%%%%%%%%%%%%%%%%%%%%%
%%%%%%%%%%%%%%%%%%%%%
%% Example comp-rea.
%%%%%%%%%%%%%%%%%%%%%
\begin{ex} \label{ex-comp-rea}
Let
\[
 \f(z)=\frac{2\pi\,i+(1-2\pi\,i)z}{1+2\pi\,i-2\pi\,i\,z} \,, \qquad
 f(z)=e^{-\dfrac12\,\dfrac{1+z}{1-z}} \,.
\]
Then $\f$ is a parabolic automorphism with Denjoy-Wolff point
$p=1$ and $f=f\circ\f$ holds in $\D$.
\end{ex}
%%%%%%%%%%%%%%%%%%%%%
\par
Note also that without the assumption on the existence of the
angular limit in (b) the result no longer holds, as the following
example shows.
%%%%%%%%%%%%%%%%%%%%%
%% Example log.
%%%%%%%%%%%%%%%%%%%%%
\begin{ex} \label{ex-log}
Consider the hyperbolic non-automorphic map $\f(z)=(1+z)/2$.
Taking the principal branch of logarithm restricted to the right
half-plane, define
\[
 f(z)=\dfrac12 \exp\(-\dfrac{\pi^2}{\log 2}\) \exp\(\dfrac{2\pi
 i}{\log 2} \log\dfrac1{1-z}\) \,.
\]
The function $f$ is easily seen to map the disk onto a compact
subset of $\D$ and is therefore an elliptic self-map of $\,\D$. It
is also readily verified that $f\circ\f=f$.
\end{ex}
%%%%%%%%%%%%%%%%%%%%%
\par
Before stating further results, let us recall that the
\textit{hyperbolic metric\/} in the disk is defined by
\[
 \vr(z,w)=\frac{1}{2} \log\frac{1+|\f_w(z)|}{1-|\f_w(z)|}=
 \inf_\g\int_\g\frac{|d\z|}{1-|\z|^2} \,,
\]
taking the infimum over all rectifiable curves $\g$ in $\D$ from
$z$ to $w$.
\par
Given a holomorphic self-map $g$ of $\D$ and a point $z_0$ in
$\D$, we define the \textit{forward orbit\/} of $z_0$ under $g$ as
the sequence $z_{n}=g_{n}(z_{0})$. It is customary to say that $g$
is of \textit{zero hyperbolic step\/} if for some point $z_{0}$
its iterations $z_{n}=g_{n}(z_{0})$ satisfy the condition
$\lim_{n\to\infty} \vr(z_{n},z_{n+1})=0$. It is well-known that
the word \textquotedblleft some\textquotedblright\ here can be
replaced by \textquotedblleft all\textquotedblright. In other
words, the definition does not depend on the choice of the initial
point of the orbit.
\par
Using the Schwarz-Pick Lemma, it is easy to see that the maps
which are not of zero hyperbolic step are precisely those
holomorphic self-maps $\f$ of $\D$ for which
\[
 \lim_{n\to\infty}\vr(z_{n},z_{n+1})>0 \,,
\]
for some forward orbit $\{z_n\}_{n=1}^\infty$ of $g$, and hence
for all such orbits. This is the reason why they are called
\textit{maps of positive hyperbolic step}. For a survey of these
properties, the reader may consult \cite{CDP2}.
\par
It is easy to show that if $g$ is elliptic and is not an
automorphism, then it is of zero hyperbolic step. If $g$ is
hyperbolic, then it is of positive hyperbolic step. The following
dichotomy holds for parabolic linear fractional maps: every
parabolic automorphism of $\D$ is of positive hyperbolic step,
while all non-automorphic, linear fractional, parabolic self-maps
of $\D$ are of zero hyperbolic step.
\par
%%%%%%%%%%%%%%%%%%%%%
%% Theorem comp-2.
%%%%%%%%%%%%%%%%%%%%%
\begin{theorem}
\label{thm-comp-2} Let $\f,\p$ be two linear fractional self-maps
of $\D$, both different from the identity, and let $f\in
\ch(\D,\D)$ be conformal at $p$, where $p$ is the Denjoy-Wolff
point of $\f$. Assume that the intertwining equation
$f\circ\f=\p\circ f$ holds. Then:
\begin{enumerate}
\item[(a)] If $\f$ is elliptic non-automorphic, then $\p$ is also
elliptic non-auto\-morphic.

\item[(b)] If $\p$ is elliptic non-automorphic, then the function
$\f$ is either a hyperbolic map or an elliptic non-automorphic
map.

\item[(c)] The map $\f$ is an elliptic automorphism if and only if
$\p$ is also an elliptic automorphism.

\item[(d)] If $\p$ is hyperbolic, then $\f$ is also hyperbolic.

\item[(e)] If $\f$ is hyperbolic, then $\p$ is either a hyperbolic
or an elliptic non-automorphic map. Moreover, if $\f$ is a
hyperbolic automorphism, then $\p$ is either a hyperbolic
automorphism or an elliptic non-auto\-morphic map such that
$\overline{\p(\D)}\cap \partial\D\neq\emptyset$.

\item[(f)] $\f$ is parabolic non-automorphic if and only if $\p$
is also parabolic non-automorphic.

\item[(g)] $\f$ is a parabolic automorphism if and only if $\p$ is
also a parabolic automorphism.
\end{enumerate}
\end{theorem}
%%%%%%%%%%%%%%%%%%%%%%%
\begin{proof}
In what follows, $q$ will always denote the Denjoy-Wolff point of
$\p$.
\par\smallskip
(a) Like in the discussion at the beginning of Subsection~\ref{subs-conform}, we see that $\p$ must also be elliptic and $q=f(p)\in \D$. Moreover, $f^{\prime}(p)\,\f^{\prime}(p)= \p^{\prime}(q)f^{\prime}(p)$. Since $f^{\prime}(p)\neq 0$, we conclude that $|\p^{\prime}(q)|=|\f^{\prime}(p)|<1$. Therefore, $\p$ is an  elliptic and non-automorphic map.
\par\smallskip
(b) We will show that $\f$ is neither a parabolic map nor an
elliptic automorphism.
\par
If $\f$ were parabolic, applying Proposition~\ref{prop-bdry}, the chain
rule for the angular derivative and conformality of $f$ at $p$, we
would obtain $\p^{\prime}(q)=\f^{\prime}(p)=1$, which is
impossible because $\p$ is elliptic and non-auto\-morphic.
Assuming that $\f$ is an elliptic automorphism, one can apply a
completely analogous but easier argument to get a contradiction.
\par\smallskip
(c) We first prove the forward implication. This case is more
delicate than the previous ones because our proof will not use the
conformality of $f$ at $p$. (Note that with that assumption a much
simpler proof is possible.) Since $\f$ is elliptic, $\p$ is also
such. Let us prove that $\p$ is also an automorphism. We have
already established earlier that $f(p)=q$. Consider the function
$g=f\circ\f_{p}$. Then $g(\la z)=\p(g(z))$ for all $z$. It readily
follows by induction that $g(\la^{n}z)= \p_{n} (g(z))$ for all
positive integers $n$ and all $z\in\D$. Observe also that
$g(0)=q$.
\par
Suppose that $\p$ is not an automorphism. Then $\p_{n}\to q$
pointwise. Thus, for each $z\in\D$,
\[
 \left\vert g(\la^{n}z)-g(0)\right\vert =\left\vert \p_{n}
 (g(z))-q\right\vert \to 0 \quad \text{as} \quad n\to\infty.
\]
Since $|\la|=1$, by continuity and a basic compactness argument it
is easy to see that for each $r\in(0,1)$, there is a point
$\xi\in\D$, with $|\xi|=r$, such that $g(\xi)=g(0)$. Therefore,
$g$ is constant, hence so is $f$, which contradicts our
assumption. Thus, $\p$ is an automorphism.
\par
Now for the reverse implication. Suppose that $\f$ is not
elliptic. Choose a sequence $\{z_n\}_{n=1}^\infty$ in the unit
disk that converges to $p$ non-tangentially. Then also $\f(z_n)\to
p\,$ non-tangentially and, using the conformality of $f$ at $p$,
we conclude that the sequence $\{f(\f(z_n))\}_{n=1}^\infty$ is
convergent to some point $\widehat{q}$. Moreover, the sequence
$\{f(z_n)\}_{n=1}^\infty$ also tends to $\widehat{q}$, hence $\p
(\widehat{q})=\widehat{q}$. Since $\widehat{q}\in \overline{\D}$
and $\p$ is an elliptic automorphism, we conclude that
$\widehat{q}$ is the Denjoy-Wolff point of $\p$. Now, applying the
chain rule in the intertwining equation \eqref{inter-eq} and using
again the conformality of $f$ at $p$, we get that $\f^{\prime}(p)=
\p^\prime(q) \in\partial\D\setminus \{1\}$, which contradicts our
assumption that $\f$ is either hyperbolic or parabolic. Therefore,
$\f$ is elliptic.
\par
Once we know that $\f$ is elliptic, we just apply the chain rule
in the intertwining equation together with the conformality of $f$
at $p$ to obtain that $\f^{\prime}(p)=\p'(q)\in \partial \D
\setminus \{1\}$. Thus, $\f$ is an automorphism.
\par\smallskip
(d) We know that
\[
 \vr (\f_{n}(0),\f_{n+1}(0))\ge \vr (f(\f_{n}(0)),f(\f_{n+1}(0)))=
 \vr (\p _{n}(f(0)),\p_{n+1}(f(0))).
\]
Since $\p$ is a hyperbolic linear fractional map, the sequence
\[
 \{\vr (\p _{n}(f(0)),\p _{n+1}(f(0)))\}_{n=1}^{\infty }
\]
converges to a positive real number. Therefore, $\f$ is of
positive hyperbolic step and this implies that $\f$ is either an
elliptic automorphism, a parabolic automorphism, or a hyperbolic
map. It cannot be elliptic since $f(p)=q$ and $\p$ is not
elliptic. Bearing in mind Proposition~\ref{prop-bdry} and using the chain
rule for the angular derivative, we obtain that $f^{\prime}(p)\,
\f^{\prime}(p)=\p^{\prime}(q)f^{\prime}(p)$. Since $f$ is
conformal at $p$, we find that $\f^{\prime}(p)=\p^{\prime}(q)\in
(0,1)$. In particular, $\f$ is not parabolic and, therefore, it is
hyperbolic.
\par\smallskip
(e) This case seems to require the most involved proof by far. We
first observe that $\p$ can neither be a parabolic map nor an
elliptic automorphism.
\par
If $\p$ were a parabolic map, then applying Proposition~\ref{prop-bdry},
the chain rule for the angular derivative, and conformality of $f$
at $p$, we would obtain that $\f^{\prime} (p)=\p^{\prime} (q)=1$,
which is absurd because $\f$ is hyperbolic.
\par
Suppose now $\p$ is an elliptic automorphism. By assumption, $f$ is
conformal at $p$, we know that there exists $\widehat{q}= \angle\lim_{z\to p} f(z)\in\overline{\D}$. Since $\f$ maps non-tangential regions at $p$ into non-tangential regions at
$p$, we deduce that $\p(\widehat{q})=\widehat{q}$. Given that $\p$ has
only one fixed point in $\overline{\D}$, it follows that $\widehat{q}=q$. Therefore, by the Denjoy-Wolff theorem
\[
 \lim_{n\to\infty} |f(\f_n(0))|=|q|=1=\lim_{n\to\infty}
 |\p_n(f(0))|\,.
\]
However, since $\p$ is an elliptic automorphism of $\D$, hence it
is a rotation in hyperbolic metric and so $\sup_{n\in\N}
|\p_n(f(0))|<1$, which is impossible.
\par
We have thus shown that $\p$ must be either a hyperbolic map or an
elliptic non-automorphic map.
\par
Our next step will be to show that if $\f$ is a hyperbolic
automorphism and $\p$ is an elliptic non-automorphic map, then
$\overline{\p(\D)}\cap\partial\D\neq\emptyset$. Suppose, on the
contrary, that $\overline{\p(\D)}\subset\D$. Denote by $q$ the
Denjoy-Wolf point of $\p$. Then replacing $\p$ by the function
$\f_{q}\circ\p\circ\f_{q}$ and $f$ by $\f_{q}\circ f$ if
necessary, we may assume that $q=0$. In this case, we have
$\p(z)=\frac{az}{cz+1}$ and, by a special case of
Proposition~\ref{prop-lft-self} or Lemma~\ref{lem-lft-self}, we
must have $|a|+|c|\leq 1$. Obviously, $a=\p^{\prime}(0)$ and then
$0<|a|<1$. By conformality of $f$ at $p$, deriving the
intertwining equation at the point $p$, we obtain $0<a<1$. If
$a+|c|=1$, then
\[
 \p\left(-\frac{\overline{c}}{|c|}\right)=
 -\frac{\overline{c}}{|c|}\in\partial\D \,.
\]
This contradicts the assumption that $\overline{\p(\D)} \subset
\D$, hence $a+|c|<1$. Write
\[
 \s(z)=\frac{z}{Cz+1}\,, \quad \text{where \ } C=c/(1-a) \,.
\]
From $a+|c|<1$ it follows that $|C|<1$, so $\s$ is bounded on the
unit disk. Moreover, $\s\circ\p=a\,\s$. Define $g:=\s\circ f$.
Then $g$ is an analytic and bounded function in the unit disk and
$g\circ\f=a g$. Since $\f$ is an automorphism, we have
\[
 \sup\{|g(z)|:z\in\D\}=\sup\{|g\circ\f(z)|:z\in\D%
 \}=a\sup\{|g(z)|:z\in\D\}.
\]
But $a<1$ and $g$ is bounded, hence $g(z)=0$ for all $z\in\D$.
This implies that $f(z)=0$, for all $z\in\D$. In other words, $f$
is constant, which is in contradiction with our assumptions.
\par
It only remains to show that if $\f$ is an hyperbolic automorphism
and $\p$ is also hyperbolic, then $\p$ is also an automorphism.
Again, let us denote by $q\in\partial\D$ the corresponding
Denjoy-Wolf point of $\p$ (now, $\p^{\prime}(q)\in(0,1))$.
Consider the rotation $r(z)=p\,\overline{q}\,z$, $z\in\D$ that
maps $q$ to $p$ (note that $r^{-1}(z)=\overline{p}\,q\,z$ is also
a rotation) and define $h:=r\circ f$ and $\widehat{\p}:=
r\circ\p\circ r^{-1}$. A straightforward computation shows that
$\widehat{\p}$ is a hyperbolic map with Denjoy-Wolff point $p$
such that $\widehat{\p}^{\prime}(p)=\p^{\prime}(q)$ and
\[
 h\circ\f=\widehat{\p}\circ h.
\]
Moreover, by the conformality of $f$ at $p$, we see that
$\widehat{\p}^{\prime}(p)=\f^{\prime}(p)$.
\par
From this point on, we change the setting from $\D$ to $\H$ and
consider the analytic functions in the right half-plane given by
$H:=T_{p}\circ h\circ T_{p}^{-1}$, $\F:=T_{p}\circ\f \circ
T_{p}^{-1}$, $\Psi:=T_{p}\circ\widehat{\p}\circ T_{p}^{-1}$.
Bearing in mind that $\widehat{\p }^{\prime}(p)=\f^{\prime}(p)$,
we see that there exist $\a>1$, $b$, $A_1$, and $A_2\in\R$ such
that $b\neq0$, $A:=A_1+iA_2\ne 0$ and $A_1\ge 0$ such that
\[
 H(\alpha w+ib)=\alpha H(w)+A,\quad\text{for all }w\in\H.
\]
In order to show that $\Psi$ is an automorphism of the right
half-plane and thus complete the proof, we need only show that
$A_1=0$.
\par
Iterating, we obtain
\[
 H\(\alpha^{n}w+ib\frac{\alpha^{n}-1}{\alpha-1}\)=\alpha^{n}H(w)+
 A\frac{\alpha^{n}-1}{\alpha-1}
\]
for all $n\in\mathbb{N}$ and for all $w\in \H$. Note that, for all
$w\in\H$, the sequence $\{\alpha^{n}w+ib\dfrac
{\alpha^{n}-1}{\alpha-1}\}_{n=1}^\infty$ tends non-tangentially to
$\infty$. Since $H$ is an analytic self-map of $\H$, by Wolff's well-known theorem (see, for example, \cite[p.~60, Exercise~2.3.10~(b)]{CM}), there exists a constant $c\geq0$ such that
\[
 c=\angle\lim_{w\to\infty}\frac{H(w)}{w}=\lim_{n\to\infty}%
 \frac{H(\alpha^{n}w+ib\frac{\alpha^{n}-1}{\alpha-1})}{\alpha^{n}%
 w+ib\frac{\alpha^{n}-1}{\alpha-1}}.
\]
Therefore, for all $w\in\H$,
\[
 c=\lim_{n\to\infty}\frac{\alpha^{n}H(w)+A \frac{\alpha^{n}-1}{\alpha
-1}}{\alpha^{n}w+ib\frac{\alpha^{n}-1}{\alpha-1}}=\lim_{n\to\infty
}\frac{H(w)+A\frac{1-\alpha^{-n}}{\alpha-1}}{w+ib\frac{1-\alpha^{-n}}
{\alpha-1}}=\frac{H(w)+\frac{A}{\alpha-1}}{w+\frac{ib}{\alpha-1}}.
\]
Hence,
\[
H(w)=cw+\frac{cb}{\alpha-1}i-\frac{A}{\alpha-1},\text{ }w\in\H.
\]
Now, since $H$ is a self-map of $\H$, we conclude that
\[
 0\leq\operatorname{Re}\left(\frac{cb}{\alpha-1}i -\frac{A}{\alpha-1}
 \right)=-\frac{A_1}{\alpha-1} \,.
\]
Our assumption $\alpha>1$ tells us that $A_1\le 0$. We already
know that $A_1\ge 0$, hence $A_1=0$, and we are done.
\par\smallskip
(f) We first prove the forward implication. Recall that a
parabolic linear fractional map which is not an automorphism is of
zero hyperbolic step. Thus, given $z\in\D$, we have that
$\vr(\f_{n}(0),\f_{n+1} (0))\to 0$ as $n\to\infty $. Hence
\[
 \vr(\f_{n}(0),\f_{n+1}(0))\ge \vr(f(\f_{n}(0)),f(\f_{n+1}(0)))=
 \vr(\p_{n}(f(0)),\p_{n+1}(f(0))).
\]
Therefore, $\lim_{n\to\infty}\vr (\p_{n}(f(0)),\p_{n+1}
(f(0)))=0$, which shows that $\p$ is of zero hyperbolic step.
Thus, $\p$ is also a parabolic or an elliptic linear fractional
map different from an automorphism. Bearing in mind statement (b),
we conclude that $\p$ is parabolic non-automorphic.
\par
Now for the reverse implication. From the statements (a), (c), and
(e), we know that $\f$ is parabolic, so we only have to prove that
it is non-automorphic. Applying Theorem~\ref{thm-comp} we deduce
that
\[
 q\p^{\prime\prime}(q) f^{\prime}(p)\frac{p}{q}= p
 \f^{\prime\prime}(p)
\]
and $\text{Re}\,q\p^{\prime\prime}(q)>0$ if and only if
$\text{Re}\,p\f^{\prime\prime}(p)>0$. By recalling the following
general fact about an arbitrary linear fractional parabolic
self-map $\f$ of $\D$: if $r$ denotes its Denjoy-Wolff point, then
always $\text{Re}\,r\f^{\prime\prime}(r)\geq 0$ and $\f$ is an
automorphism if and only if $\text{Re}\,r\f^{\prime\prime}(r)=0$,
the proof is complete.
\par\smallskip
Part (g) is a trivial consequence of the remaining six statements.
\end{proof}
%%%%%%%%%%%%%%%%%%%%%%%
\par\smallskip
Observe that in parts (c) and (g) the maps $\f$ and $\p$ must have
the same type. One would expect a similar statement to hold for
hyperbolic automorphisms. However, this is false.
%%%%%%%%%%%%%%%%%%%%%
%% Example hyp-aut.
%%%%%%%%%%%%%%%%%%%%%
\begin{ex} \label{ex-hyp-aut}
Let
\[
 f(z)=\frac{i z+1}{-z+2+i}\,, \quad \f(z)=\frac{z+1}2\,, \quad \p(z)=
 \frac{(2+i) z+i}{z+1+2 i}\,.
\]
Then the intertwining equation holds, $\f$ is hyperbolic and
non-automorphic, while $\p$ is a hyperbolic automorphism.
\end{ex}
%%%%%%%%%%%%%%%%%%%%%
\par
Theorem~\ref{thm-comp-2} shows a great degree of symmetry
concerning the dynamical types of $\f$ and $\p$ in the
non-automorphic cases as well, except for the statements (b) and
(e) where a mixture of types is allowed. The following examples
tell us that indeed all combinations are possible.
%%%%%%%%%%%%%%%%%%%%%
%% Example hyp-ell.
%%%%%%%%%%%%%%%%%%%%%
\begin{ex} \label{ex-hyp-ell}
Consider the mappings
\[
 f(z)=(1-z)/2\,, \quad \f(z)=(1+z)/2\,, \quad
\p(z)=z/2\,, \quad z\in\D\,.
\]
Then $\f$ is hyperbolic and non-automorphic with Denjoy-Wolff
point $1$ and $\p$ is elliptic non-automorphic with Denjoy-Wolff
point $0$.
\end{ex}
%%%%%%%%%%%%%%%%%%%%%
It should also be mentioned that even if $\f$ is a hyperbolic
automorphism, $\p$ can be an elliptic map.
%%%%%%%%%%%%%%%%%%%%%
%% Example hyp-aut-ell.
%%%%%%%%%%%%%%%%%%%%%
\begin{ex} \label{ex-hyp-aut-ell}
Let
\[
 \f(z)=\frac{3z+1}{z+3}\,, \quad \p(z)=\frac{z}{2-z},\quad
 f(z)=\frac {1-z}{z+3}\quad\text{for all }z\in\D.
\]
Then $\f$ is a hyperbolic automorphism, $\p$ is elliptic, and
$f\circ\f=\p\circ f$. Of course, in accord with
Theorem~\ref{thm-comp-2}, $\overline{\p(\D)}\cap\partial\D
=\{1\}\neq \emptyset$.
\end{ex}
%%%%%%%%%%%%%%%%%%%%%
\par
Conformality is not used to the full extent in all cases in
Theorem~\ref{thm-comp-2}. For instance, in the forward implication
in (c), it is not needed at all.

%%%%%%%%%%%%%% A RIGIDITY PRINCIPLE AND FURTHER RESULTS %%%%%%%%%%%%%%
\section{A rigidity principle and further results on the intertwining
equation}
 \label{sec-fcn-eqns}
%%%%%%%%%%%%%%%%%%%%%%%%%%%%%%%%%%%%%%%%%%%%%%%%%%%%%%%%%%%%%%%%
The reader should be warned that \eqref{inter-eq} does not always
imply that $f$ must be an LFT (that is, we do not necessarily have
the rigidity principle here), as the following example shows.
%%%%%%%%%%%%%%%%%%%%%
%% Example non-univ.
%%%%%%%%%%%%%%%%%%%%%
\begin{ex} \label{ex-non-univ}
Let $\f(z)=z/2$, $\p(z)=z/4$, and $f(z)=z^2$. It is clear that
$f\circ\f=\p\circ f$ but $f$ is not and LFT; it is not even
univalent.
\end{ex}
%%%%%%%%%%%%%%%%%%%%%
Thus, again some additional conditions on $f$ (\textit{e.g.},
local univalence near the Denjoy-Wolff point of $\f$) should be
required in order to get the rigidity. We now analyze all possible
cases.

%%%%%%%%%%%%%%%%%%%%%%%
\subsection{A rigidity principle for intertwining}
\label{subs-pfs-rigid}
%%%%%%%%%%%%%%%%%%%%%%%
The following statement covers the majority of the possible cases regarding the dynamical type of $\f$. The remaining two automorphic cases shall be dealt with in separate subsections.
\par
%%%%%%%%%%%%%%%%%%%%%
%% Theorem inter.
%%%%%%%%%%%%%%%%%%%%%
\begin{theorem}
 \label{thm-inter-rigid}
Let $\f$ be a linear fractional self-map of $\,\D$ which is either
elliptic and non-automorphic, parabolic and non-automorphic, or hyperbolic. Let $p$ be the Denjoy-Wolff point of $\f$ and $f$ a
self-map of $\D$ conformal at $p$. If \eqref{inter-eq} holds for
some linear fractional self-map $\p$ of $\D$ then $\p\neq id_\D$ and $f$ is also a linear fractional map.
\end{theorem}
%%%%%%%%%%%%%%%%%%%%%
\par\smallskip
%%%%%%%%%%%%%%%%%%%%%%%
\begin{proof} First of all, observe that if \eqref{inter-eq} is satisfied then part (b) of Theorem~\ref{thm-comp-1} tells us that the case when $\p$ is the identity map is excluded here.
\par
For the sake of simplicity, suppose that the map $\f$ is either parabolic and non-automorphic or hyperbolic. When $\f$ is
elliptic and non-automorphic, the proof carries through by replacing everywhere the words ``converges non-tangentially'' by ``converges''.
\par
Since the sequence $(\f_n(0))$ converges non-tangentially to the
point $p$, it follows that $f^\prime(\f_n(0))\to f^\prime(p)$ as $n\to \infty$.
\par
Let
\[
\f_n(z)=\frac{\a_n z +\b_n}{\g_n z+\d_n}\,, \quad \p_n(z)=\frac{a_nz +b_n}{c_n z+d_n}\,,
\]
where
\[
\a_n\d_n-\b_n \g_n=1\,, \quad a_n d_n - b_n c_n=1\,.
\]
Differentiation of $f\circ\f_n=\p_n\circ f$ yields $f^\prime (\f_n(z)) \f^\prime_n (z) = \p^\prime_n(f(z))f^\prime(z)$ and therefore
\begin{equation}
 \label{eq_thm-inter}
 \frac{f^\prime(\f_n(z))}{(\g _n z+\d_n)^2}=\frac{f^\prime(z)}{(c_n f(z)+d_n)^2}
\end{equation}
for all $n$ and for all $z\in \D$.
\par
Write $c_n^*:=\frac{c_n}{|c_n|+|d_n|}$, $d_n^*:=\frac{d_n}{|c_n|+|d_n|}$,
$\g_n^*:=\frac{\g_n}{|\g_n|+|\d_n|}$, and $\d_n^*:=\frac{\d_n}{|\g_n|+|\d_n|}$.
By passing to convergent subsequences if necessary, we obtain $c_{n_k}^*\to c^*$, $d_{n_k}^*\to d^*$, $\d_{n_k}^*\to \d^*$, and $\g_{n_k}^*\to \g^*$, where $c^*,d^*,\g^*,\d^*\in\overline{\D}$. Since $|c_n^*|+|d_n^*|=1$ and $|c_n^*|<|d_n^*|$ for all $n$, we also get $|c^*|+|d^*|=1$ and $|c^*|\leq|d^*|$. Similarly, $|\g^*|+|\d^*|=1$ and $|\g^*|\leq|\d^*|$. Therefore, $\g^*z+\d^*\neq 0$ and $c^*z+d^*\neq 0$ for all $z\in\D$.
\par
Choose $z_0\in\D$ such that $f(z_0)\neq 0$ and $f^\prime(z_0)\neq 0$. By (\ref{eq_thm-inter}), $f^\prime(\f_n(z_0))\neq 0$ for all $n$. Using (\ref{eq_thm-inter}) for $z$ and $z_0$, we have
$$
\frac{f^\prime(\f_n(z))}{f^\prime(\f_n(z_0))}\left( \frac{\g _n z_0+\d_n}{\g _n z+\d_n} \right)^2
=\frac{f^\prime(z)}{f^\prime(z_0)}\left( \frac{c _n f(z_0)+d_n}{c_n f(z)+d_n} \right)^2.
$$
Thus
$$
\frac{f^\prime(\f_n(z))}{f^\prime(\f_n(z_0))}\left( \frac{\g _n^* z_0+\d_n^*}{\g _n^* z+\d_n^*}
\right)^2 =\frac{f^\prime(z)}{f^\prime(z_0)}\left( \frac{c _n^* f(z_0)+d_n^*}{c_n^* f(z)+d_n^*} \right)^2.
$$
Replacing $n$ by $n_k$, recalling that $f$ is conformal at $p$ and taking limits, we get
$$
\left( \frac{\g ^* z_0+\d^*}{\g ^* z+\d^*} \right)^2 =\frac{f^\prime(z)}{f^\prime(z_0)}\left( \frac{c ^*
f(z_0)+d^*}{c^* f(z)+d^*} \right)^2.
$$
That is, there are complex numbers $A$, $B$, $C$, and $D$ with $|C|\leq |D|$ such that
$$
f^\prime(z)=\left( \frac{Af(z)+B}{Cz+D} \right)^2 \qquad \textrm{for\  all} \quad z\in\D.
$$
Let $a=f(0)$ and consider $g=\f_a\circ f$. A simple computation shows that there exist $A^*,B^*$ such that
$$
g'(z)=\left( \frac{A^*g(z)+B^*}{Cz+D} \right)^2 \qquad \textrm{for\  all} \quad z\in\D.
$$
That is, $g$ is the unique solution of the Cauchy problem $y^\prime=\left( \frac{A^*y+B^*}{Cz+D} \right)^2$ with
$y(0)=0$. Thus, $g(z)=\dfrac{{B^*}^2 z}{(DC-B^*A^*)z+D}$. Since $f=\f_a\circ g$, it follows that $f$ must also be a linear fractional map.
\end{proof}
%%%%%%%%%%%%%%%%%%%%%%%
\par\smallskip
A few comments and examples are in order. Our earlier Example~\ref{ex-non-univ} tells us that, in the elliptic case, the assumption $f^\prime(p)\neq 0$ is essential in order for the conclusion of the theorem to hold. Moreover, the assumption $f^\prime(p)\neq\infty$ cannot be omitted when the Denjoy-Wolff point of $\f$ is on the boundary of the unit disk.
%%%%%%%%%%%%%%%%%%%%%
%% Example noncom-infder.
%%%%%%%%%%%%%%%%%%%%%
\begin{ex} \label{ex-noncom-infder}
Let
\[
 \f(z)=\frac{5z+3}{3z+5} \,, \quad \p(z)=\frac{3z+1}{z+3} \,, \quad f(z)=
 \frac{\sqrt{\frac{1+z}{1-z}}-1}{\sqrt{\frac{1+z}{1-z}}+1} \,.
\]
It can be checked that \eqref{inter-eq} holds, the Denjoy-Wolff
point of $\f$ is $p=1$, while $\angle\lim_{z\to 1} f(z)=1$ and
$f^\prime(1)=\infty$, even though $f$ is not an LFT.
\end{ex}
%%%%%%%%%%%%%%%%%%%%%
The above example is easier to understand after transferring all
the maps to the right half-plane by means of conjugation $g=T\circ
f\circ T^{-1}$ via the Cayley map $T(z)=(1+z)/(1-z)$. In the right
half-plane \eqref{inter-eq} becomes simply $g(4w)=2 g(w)$, where
$g(w)=\sqrt{w}$ is the map that corresponds to $f$ with an
appropriately defined branch of the square root. Notice that in
this example the function $\f$ is hyperbolic and, by
Theorem~\ref{thm-comp}, we know that this situation cannot occur when $\f$ and $\p$ are parabolic.
\par
The following example shows why it was important to exclude the elliptic automorphisms from Theorem~\ref{thm-inter-rigid}.
%%%%%%%%%%%%%%%%%%%%%
%% Example non-autom.
%%%%%%%%%%%%%%%%%%%%%
\begin{ex} \label{ex-non-autom}
Consider the following self-maps of the disk:
\[
 \f(z)=\p(z)=-z \,, \qquad f(z)=\frac{z+z^3}2 \,.
\]
Note that $f(\D)\subset\D$, $p=0$, $f^\prime(0)\neq 0$, and
\eqref{inter-eq} still holds, even though $f$ is clearly not an
LFT.
\end{ex}
%%%%%%%%%%%%%%%%%%%%%
\par
We also point out that there is a self-map of $\D$ which commutes
with a parabolic automorphism but is far from being an LFT. This
has been known for some time \cite[Proposition~1.2.6]{A}.

%%%%%%%%%%%%%%%%%%%%%%%
\subsection{The case when $\f$ is an elliptic automorphism}
\label{subs-proofs-ell-aut}
%%%%%%%%%%%%%%%%%%%%%%%
The following characterization seems to be new, in the sense that
it generalizes Proposition~1.2.26 from \cite{A}. It should be
remarked that similar questions for intertwining were raised
explicitly and studied in \cite{PC} in the context of the
classical semi-conjugation. The reader should note the dichotomy
between the conjugations related to the roots of unity and all the
remaining ones.
\par
%%%%%%%%%%%%%%%%%%%%%
%% Theorem inter-ell-aut.
%%%%%%%%%%%%%%%%%%%%%
\begin{theorem} \label{thm-inter-rigid-ell-aut}
Let $\f$ be an elliptic automorphism (different from the identity)
with Denjoy-Wolff point $p$ and $\la=\f^\prime(p)$. Let $\p$ be an
arbitrary linear fractional self-map of $\D$. Suppose $f$ is a
non-constant self-map of $\,\D$ such that $f\circ\f=\p\circ f$ and
define $n_{0}=\min \{n\in\N\,\colon\,f^{(n)}(p)\neq0\}$.
\par
\begin{enumerate}
\item[(i)] If $\p$ is the identity map and $m_0=\min\{n\ge 1
\,\colon\,\la^{n}=1\}$ as in the proof of
Theorem~\ref{thm-comp-1}, then there is a non-constant self-map of
the unit disk $g$, with $g(0)=0$, such that
\[
(\f_{f(p)} \circ f \circ \f_p) (z)=g(z^{m_0}), \quad z\in \D.
\]
Moreover, the function $f$ is never a linear fractional map.
\item[(ii)] Assume $\p$ is not the identity map, $q$ is its
Denjoy-Wolff point, and $\mu=\p^\prime (q)$. Then
$\la^{n_{0}}=\mu$. Also, the following dichotomy takes place:
\begin{enumerate}
 \item If $\la^{n}\neq1$ for every positive integer $n$, then
there exists a point $\b$ in $\overline{\D}\setminus\{0\}$ such
that $(\f_q\circ f\circ\f_p)(z)=\b z^{n_0}$ for all $z\in\D$. In
particular, $f$ is a linear fractional map if and only if
$n_{0}=1$.
 \item If $\la^{n}=1$ for some positive integer $n$, and the
integer $m_0$ is defined as above, then there exists a non-zero
self-map $g$ of $\,\D$ such that
\[
 (\f_q\circ f\circ\f_p)(z)= z^{n_0} g(z^{m_0}), \quad z\in\D .
\]
In particular, $f$ is a linear fractional map if and only if $g$
is constant and $n_0=1$.
\end{enumerate}
\end{enumerate}
\end{theorem}
%%%%%%%%%%%%%%%%%%%%%
\par
%%%%%%%%%%%%%%%%%%%%%
\begin{proof}
Keeping in  mind Theorem~\ref{thm-comp-1}, one sees that the proof
of statement (i) is similar to the proof of part (b) in statement
(ii) by setting $n_0=0$. Therefore, we will only present in detail
the proof of the latter case. It is easy to fill in the
corresponding arguments in (i).
\par
Since $\f$ is an elliptic automorphism, it follows by part (c) of
Theorem~\ref{thm-comp-2} that $\p$ is also such. (As remarked in
the proof of that part of the theorem, the assumption on
conformality of $f$ is not needed precisely in this implication.)
\par
Each of the maps $\f_p\circ\f\circ\f_p$ and $\f_q\circ\p\circ\f_q$
is a disk automorphism and fixes the origin, hence a rotation:
$\f_{p}(\f(\f_{p}(z)))=\la z$ and $\f_{q}(\p(\f_{q}(z)))=\mu z$,
for the values $\la$, $\mu$ as defined in the statement of the
theorem and for all $z$ in $\D$. Considering the self-map
$h=\f_{q}\circ f\circ\f_{p}$ of $\D$ which fixes the origin, the
intertwining equation implies that $h(\la z)=\mu h(z)$ for all $z$
in the disk. Let
\[
 h(z)= {\displaystyle\sum\limits_{n=n_{0}}^{\infty}} a_{n}z^{n},
\]
with $a_{n_{0}}\neq0$. Thus we have $a_{n}\la^{n}=a_{n}\mu$ for
all $n\geq n_{0}$. In particular, $\la^{n_{0}}=\mu$ (because
$a_{n_{0}}\neq0$), which completes the proof of the first part of
the assertion.
\par
We now consider the two cases corresponding to the statements (a)
and (b). Observe first that if $a_{n}\neq 0$ then $\la^{n}=\mu
=\la^{n_{0}}$. If there is no $m\neq0$ such that $\la^{m}=1$, then
$n=n_{0}$ since $\la^{n-n_{0}}=1$. This implies that
$h(z)=a_{n_{0}}z^{n_{0}}$, with $0<|a_{n_{0}}|\le 1$ because $h$
is a self-map of the unit disk. This implies (a).
\par
If $\la^{m}=1$ for some non-negative integer $m$, there exists a
non-negative integer $k$ such that $m-n_{0}=k m_{0}$, where $m_0$
is the number defined in the statement of the theorem. Hence there
is a sequence $\{b_{k}\}_{k=0}^{\infty}$ such that $h(z)=
\sum_{k=0}^{\infty} b_{k}z^{n_{0}+km_{0}}$ for all $z$. Then
\[
 \f_{q}(f(z))= \sum_{k=0}^{\infty} b_{k}\f_{p} (z)^{n_{0}+km_{0}} =
 \f_p^{n_0}(z) g(\f_p^{m_0}(z)) \,,
\]
where $g$ denotes the function analytic in $\D$ whose Taylor
coefficients are $b_k$. This equation implies part (b).
\end{proof}
%%%%%%%%%%%%%%%%%%%%%

%%%%%%%%%%%%%%%%%%%%%%%
\subsection{The case when $\f$ is a parabolic automorphism}
\label{subs-proofs-parab-aut}
%%%%%%%%%%%%%%%%%%%%%%%
Our last result about intertwining deals with parabolic
automorphisms. In what follows, we will write $\D^\ast$ for the
punctured unit disk $\D \setminus\{0\}$. Also, $\ch (\D^\ast,\D)$
will denote the collection of all analytic functions from the
punctured disk into the unit disk and $\ch (\D^\ast,\C)$ the set
of all functions which are analytic in the punctured disk.
%%%%%%%%%%%%%%%%%%%%%
%% Theorem inter-ell-aut.
%%%%%%%%%%%%%%%%%%%%%
\begin{theorem} \label{thm-inter-rigid-par-aut}
Let $\f$ be a parabolic automorphism with Denjoy-Wolff point $p$,
$\p$ an arbitrary linear fractional self-map of $\,\D$, and $f$  a
self-map of $\,\D$ which is conformal at the point $p$ and such
that $f\circ\f=\p\circ f$.
\begin{enumerate}
 \item[(a)]
If $\p$ is the identity map, then there is a map $g \in \ch
(\D^\ast,\D)$ such that
\[
 f(z)=g\(\exp\(-\frac{2\pi}{|\f^{\prime\prime}(p)|}T_p(z)\)\)\,, \quad z\in\D.
\]
The function $f$ cannot be a linear fractional map.
 \item[(b)] If
$\p$ is not the identity map, denote by $q$ its Denjoy-Wolff point
and write
$\la=\frac{q\p^{\prime\prime}(q)}{p\f^{\prime\prime}(p)}$. Then
$\la\in (0,+\infty)$ and there is a map $g\in\ch (\D^\ast,\C)$
with $g(\D^\ast)\subset\overline{\H}$ such that
\[
 (T_q\circ f)(z)= \la T_p(z)+ g\(\exp\(-\frac{2\pi}{|\f^{\prime\prime}
 (p)|}T_p(z)\)\)\,, \quad z\in\D.
\]
The function $f$ is a linear fractional map if and only if $g$ is
constant.
\end{enumerate}
\end{theorem}
%%%%%%%%%%%%%%%%%%%%%
\begin{proof}
We omit the proof of part (a) as its idea is essentially contained
in the proof of (b) given below; just work with $\la =0$.
\par
(b) By Theorem~\ref{thm-comp-2}, we know that $\p$ is also a
parabolic automorphism. Set $g(w)=T_{q}\circ f\circ
T_{p}^{-1}(w)$, $w\in \H$. Clearly, $g$ is a holomorphic self-map
of $\H$ so there exists  $c=\angle \lim_{w\to \infty
}\dfrac{g(w)}{w}\in \lbrack 0,+\infty )$. Moreover, since $f$ is
conformal at $p$, we have that $c>0$ and indeed (see the proof of
Theorem~\ref{thm-comp}) $c=\dfrac{q}{p}\dfrac{\p^{\prime\prime}
(q)}{\f^{\prime\prime}(p)}=\la $.
\par
Wolff's Theorem implies
\[
 c=\inf \left\{ \frac{\text{Re}\, g(w)}{\text{Re}\, w}:w\in
 \H\right\}
\]
and the infimum is attained if and only if $g(w)=cw+i\beta $, for
some real number $\beta $. Note that, in this case, the result
follows by simply considering $g$ to be identically equal to
$i\,\beta $. Thus, let us assume that the above infimum is not
attained. In this situation, $h(w)=g(w)-\la w$ is a holomorphic
self-map of $\H$. We note that the intertwining equation for $f$
transfers to an equation in $g$ in the following way: $g\circ
\widehat{\f}= \widehat{\p}\circ g$, where
\begin{eqnarray*}
 \widehat{\f}(w) &=&T_{p}\circ \f \circ T_{p}^{-1}(w)=w+ia \,, \\
 \widehat{\p}(w) &=&T_{q}\circ \f \circ T_{q}^{-1}(w)=w+ib \,,
\end{eqnarray*}
and $a$ and $b$ are non-zero real numbers. Bearing in mind that
$\la =b/a$, a computation shows that $ h(w+ina)=h(w)$, for all 
$n\in \mathbb{Z}$. In other words, $h$ is automorphic under the 
group $\Gamma $ generated by $\widehat{\f }$. Since $\widehat{\f
}$ is parabolic, the Riemann surface $\H/\Gamma $ is biholomorphic to $\D^{\ast }=\mathbb{D\setminus \{}0%
\mathbb{\}}$ (see \cite[page 24]{A}). Moreover, it is well-known
that
\[
 \pi (w)=\exp \left( -\frac{2\pi }{|a|}w\right) \,, \quad w\in
 \H\,,
\]
defines a covering map $\pi$ from $\H$ onto $\D^{\ast}$ such that
$\pi\circ\widehat{\f }=\pi$. Hence, we can define\ a holomorphic
map $g$ from $\D^{\ast }$ onto $\H$ such that $h=g \circ \pi $.
Finally, note that $ |a|=|p\f
^{\prime\prime}(p)|=|\f^{\prime\prime}(p)|$ so that
\[
 T_{q}\circ f\circ T_{p}^{-1}(w)=\la w+g \left( \exp \left(
 -\frac{ 2\pi }{|\f^{\prime\prime}(p)|}w\right) \right) \,,
\]
as desired.
\par
Trivially, if $g$ is constant, then $f$ is a linear fractional
map. On the other hand, assume that $g$ is not constant and $f$ is
also a linear fractional map. In this case, $\s=T_{q}\circ f\circ
T_{p}^{-1}$ is a linear fractional map in $\H$ that fixes the point at $\infty$. Therefore, $\s(w)=d w+k$, for some $d\geq 0$ and $k\in
\overline{\H}$. Since $g$ is not constant, we must have $d\neq
\la$ so the function defined by $u(w)=\s(w)-\la w$ is univalent in
$\H$. This implies that $g \circ \pi $ is univalent, which is
clearly false. This ends the proof.
\end{proof}
%%%%%%%%%%%%%%%%%%%
\par\smallskip
%%%%%%%%%%%%%%%%%%%%%%%
\subsection{The intertwining equation and the sets of fixed points}
\label{subs-fixed-pts}
%%%%%%%%%%%%%%%%%%%%%%%
We have already seen that in most cases a rigidity principle holds: a self-map $f$ of $\D$ which solves the intertwining equation $f\circ\f=\p\circ f$ must also be a linear fractional map. However, this condition which is so often necessary, is never sufficient by itself for \eqref{inter-eq} to hold.
\par
We now state and prove, for all possible cases, necessary and
sufficient conditions for a linear fractional self-map $f$ of $\D$
to satisfy the intertwining equation $f\circ\f=\p\circ f$ for two
fixed linear fractional self-maps of the unit disk $\f$ and $\p$.
For this, some additional notation will be useful. Given a linear
fractional self-map $h$ of $\D$, we denote by $\mathrm{Fix}(h)$
the collection of all fixed points of $h$ viewed as a bijective
map of the Riemann sphere $\widehat{\C}$. That is,
$\mathrm{Fix}(h)=\{w\in\widehat{\C}:h(w)=w\}$. We note that,
whenever $h$ is not the identity, $\mathrm{Fix}(h)$ must be either
a singleton or a set consisting of two points.
%%%%%%%%%%%%%%%%%%%%%
\begin{theorem} \label{thm-last}
Let $\f,\p,f$ be three linear fractional self-maps of the unit
disk $\D$. Assume that both $\f$ and $\p$ are different from the
identity and denote by $p$ and $q$ their respective Denjoy-Wolff
points.
\begin{enumerate}
 \item
If $\f$ is elliptic, the intertwining equation $f\circ\f=\p\circ f$ holds if and only if $f(\mathrm{Fix} (\f))=\mathrm{Fix}(\p)$ and $\f^{\prime}(p)=\p^{\prime}(q)$.
 \item
If $\f$ is hyperbolic, the equation $f\circ\f=\p\circ f$ takes place if and only if $f(\mathrm{Fix} (\f))=\mathrm{Fix}(\p)$, $f(p)=q$, and $\f^{\prime}(p)=\p^{\prime}(q)$.
 \item
If $\f$ is parabolic, the equation $f\circ\f=\p\circ f$ holds if and only if $f(\mathrm{Fix} (\f))=\mathrm{Fix}(\p)$ and $f^{\prime}(p)\p^{\prime\prime }(p)=\f^{\prime\prime}(p)$.
\end{enumerate}
\end{theorem}
%%%%%%%%%%%%%%%%%%%%%
\begin{proof}
First of all, we prove that if the intertwining equation holds
then always $f(\mathrm{Fix}(\f))=\mathrm{Fix}(\p)$, $f(p)=q$,
$\f^{\prime }(p)=\p^{\prime}(q)$, and, in the parabolic case, also
$f^{\prime}(p) \p^{\prime\prime}(p)= \f^{\prime\prime}(p)$.
\par
An easy computation shows that $f(\mathrm{Fix}(\f))\subseteq
\mathrm{Fix}(\p)$. On the one hand, if $\mathrm{Fix}(\p)$ has two
points, then $\p$ is either elliptic or hyperbolic so, by Theorem
3, we conclude that $\f$ is also elliptic or hyperbolic and,
therefore, $\mathrm{Fix} (\f)$ also has two points. Since $f$ is
univalent, we deduce that $f(\mathrm{Fix}(\f))$ also has two
points, whence $f(\mathrm{Fix} (\f))=\mathrm{Fix}(\p)$. On the
other hand, if $\mathrm{Fix}(\p)$ has just one point, and since
$f(\mathrm{Fix}(\f))$ is non-empty, we obtain that
$f(\mathrm{Fix}(\f))=\mathrm{Fix}(\p)$.
\par
If $\p$ is an elliptic automorphism, then so is $\f$ (see
Theorem~\ref{thm-comp-2}). Therefore, in this case, we clearly
have $f(p)=q$. In the remaining cases, by Proposition~\ref{prop-bdry} and
the continuity of $f$ at $p$, we also conclude that $f(p)=q$.
\par
By taking the derivatives of both sides of the intertwining
equation and evaluating them at $p$, we get $f^{\prime}(p)
\f^{\prime}(p)=\p^{\prime}(f(p))f^{\prime}(p)=\p^{\prime}(q)\,
f^{\prime}(p)$. Since $f$ is conformal at $p$, it follows that
$\f^{\prime}(p)=\p^{\prime}(q)$. If $\f$ is parabolic, then so is
$\p$ (see again Theorem~\ref{thm-comp-2}). By differentiating the
intertwining equation twice, evaluating at $p$, and using the
assumption that $\f$ and $\p$ are parabolic, we arrive at the
formula $f^{\prime}(p)\p^{\prime\prime}(p)=\f^{\prime\prime}(p)$.
\par
We now prove the reverse implication. Suppose initially that $\f$
is either elliptic or hyperbolic. By assumption,
$f(\mathrm{Fix}(\f))=\mathrm{Fix} (\p)$. Since $f$ is injective,
we find that $\mathrm{Fix}(\p)$ has two points. In particular,
$\p$ is either elliptic or hyperbolic. If $\f$ is elliptic, then
$f(p)\in\D$ and it is a fixed point of $\p$. Thus, it follows that
$f(p)=q$. If $\f$ is hyperbolic, it is immediate from the
assumption of the theorem that $f(p)=q$.
\par
Denote by $p_r$ and $q_r$ the remaining fixed points of $\f$ and
$\p$ respectively. Note that $p_{r}$ and $q_{r}$ belong to
$\widehat{\C}\setminus\D$. Set $\beta(z):=z-p$ if $p_{r}=\infty$
and $\beta(z):=\dfrac{z-p}{z-p_{r}}$ if $p_{r}\neq\infty$ and
consider the linear fractional map $H=\beta\circ
f^{-1}\circ\p^{-1}\circ f\circ\f\circ\beta^{-1}$. One can easily
check that $H(0)=0$ and $H(\infty)=\infty$. Therefore, there
exists $\lambda\in\C$, $\lambda \neq0$, such that $H(z)=\lambda
z$. After a differentiation, we obtain
\begin{align*}
 \la & =H^{\prime}(0)=(f^{-1}\circ\p^{-1}\circ
 f\circ\f)^{\prime }(p)=\frac{(f\circ\f)^{\prime}(p)}{(\p\circ
 f)^{\prime}[(\p\circ
 f)^{-1}\circ(f\circ\f)(p)]}\\
 & =\frac{(f\circ\f)^{\prime}(p)}{(\p\circ f)^{\prime}(p)}
 =\frac{f^{\prime}(p)\f^{\prime}(p)}{\p^{\prime}(q)f^{\prime}(p)}
 =\frac{\f^{\prime}(p)}{\p^{\prime}(q)}=1\,.
\end{align*}
Hence $\la=1$, meaning that $H$ is the identity. After some simple
algebraic computations, we deduce that the intertwining equation
holds.
\par
Finally, we consider the remaining case when $\f$ is parabolic.
Since $\mathrm{Fix}(\f)$ has just one point, it is clear that
$\mathrm{Fix}(\p)$ has a unique point and, in particular, $\p$ is
parabolic and $f(p)=q$. Now, consider the linear fractional maps
\[
 \widehat{\f}=T_{p}\circ\f\circ T_{p}^{-1}\,, \quad
 \widehat{\p}=T_{q}\circ\p\circ T_{q}^{-1}\,, \quad F=T_{q}\circ
 f\circ T_{p}^{-1}\,.
\]
These three functions leave invariant the right-half plane and fix
the point $\infty$ on the Riemann sphere. Since $\f$ and $\p$ are
parabolic, there exist two constants $a$ and $b$ such that
$\widehat{\f}(w)=w+a$ and $\widehat{\p}(w)=w+b$. Moreover, there
must exist two other constants $c$ and $d$ such that $F(w)=cw+d$.
Then
\[
 c=\lim_{w\rightarrow\infty}\frac{F(w)}{w}=\lim_{z\rightarrow
 p}\frac {T_{q}(f(z))}{T_{p}(z)}=\lim_{z\rightarrow
 p}\frac{q+f(z)}{q-f(z)}\frac
 {p-z}{p+z}=\frac{q}{p}\frac{1}{f^{\prime}(p)}\,.
\]
From the equation $w+a=T_{p}\circ\f\circ T_{p}^{-1}(w)$, writing
$w=T_{p}(z)$, we obtain that
\[
 (p+z)(p-\f(z))+a(p-z)(p-\f(z))=(p+\f(z))(p-z)\,.
\]
By substituting $z=p$ into the second derivative of both sides of
this equation, we find out that $a=p\f^{\prime\prime}(p)$. In a
similar way, $b=q\p^{\prime\prime}(q)$. Therefore, by assumption,
we deduce that
\[
 c=\frac{q}{p}\frac{1}{f^{\prime}(p)}=\frac{q}{p}\frac{\p^{\prime\prime}
 (p)}{\f^{\prime\prime}(p)}=\frac{b}{a}\,.
\]
Moreover, $F(\widehat{\f}(w))=F(w+a)=c(w+a)+d=cw+ca+d$ and
$\widehat {\p}(F(w))=F(w)+b=cw+d+b$. Since $ca=b$, we conclude
that $F\circ \widehat{\f}=\widehat{\p}\circ F$. After transferring
this equation to the unit disk, we see that the intertwining
equation $f\circ\f=\p\circ f$ holds.
\end{proof}
%%%%%%%%%%%%%%%%%%%%%
\par
It is important to underline that, in the above theorem, all the
hypotheses are necessary. More precisely, it is not enough to
assume that $f(\mathrm{Fix} (\f))=\mathrm{Fix}(\p)$ in order to
have the intertwining equation (compare this with the commuting
case given in the beginning of Subsection~\ref{subs-proofs-commut}
below). Some examples are clearly in order here.

%%%%%%%%%%%%%%%%%%%%%
%% Example noncom-infder.
%%%%%%%%%%%%%%%%%%%%%
\begin{ex} \label{ex-3-lfts}
For the elliptic case, choose $\f(z)=z/2$, $\p(z)=z/4$, and
$f=id_{\D}$. Then $f(\mathrm{Fix}(\f))=\mathrm{Fix}(\p)=
\{0,\infty\}$ but the intertwining equation does not hold.
\par
In the hyperbolic case, consider the functions
\[
 \f(z)=\dfrac{3z+1}{z+3}\,, \quad \p_{\lambda}(z)=\dfrac
 {(\lambda+1)z+\lambda-1}{(\lambda-1)z+\lambda+1} \quad (\lambda\neq
 2)\,, \quad f=id_{\D}\,.
\]
Then $f(\mathrm{Fix}(\f))=\mathrm{Fix}(\p_{\lambda})=\{1,-1\}$ and
again the intertwining equation fails. Notice that, when
$\lambda=4\,$, the Denjoy-Wolff point of both $\f$ and $\p$ is 1
but $\f^{\prime}(1)\neq\p^{\prime}(1)$. Moreover, when
$\lambda=1/2$, the Denjoy-Wolff points of $\f$ and $\p$ are $1 $
and $-1$, respectively, and we have
$\f^{\prime}(1)=\p^{\prime}(-1)$.
\par
Finally, in the parabolic case, define
\[
 \f(z)=\dfrac{(2-i)z+i}{-iz+2+i}\,, \quad
 \p(z)=\dfrac{(1-i)z+i}{-iz+1+i}\,, \quad f=id_{\D}\,.
\]
Now we have that $f(\mathrm{Fix}(\f))=\mathrm{Fix}(\p)=\{1\}$ and
once again the intertwining equation does not hold.
\end{ex}
%%%%%%%%%%%%%%%%%%%%%

%%%%%%%%%%%%%%% APPLICATIONS OF MAIN RESULTS %%%%%%%%%%%%%%%
\section{Some applications}
 \label{sec-applns}
%%%%%%%%%%%%%%%%%%%%%%%%%%%%%%%%%%%%%%%%%%%%%%%%%%%%%%%%%%%%%%%%

We now give several applications of the results obtained in
Section~\ref{sec-compat-fcn-eqns} and Section~\ref{sec-fcn-eqns}
or of the methods used in their proofs.

%%%%%%%%%%%%%%%%%%%%%%%
\subsection{Commutation}
\label{subs-proofs-commut}
%%%%%%%%%%%%%%%%%%%%%%%
Our first application is the description of the commutant of a
linear fractional self-map of $\D$ other than certain exceptional
automorphisms. Recall that the \emph{commutant\/} of such $\f$ is
the set of all holomorphic self-maps $f$ of $\D$ such that
$f\circ\f=\f\circ f$.
\par
It seems that until now the answer to the commutant question has
been known only in a handful of cases. The following statements
are classical (see \cite[pp.~68--69, Theorem~4.3.5 and
Theorem~4.3.6]{Bea}):
\par
- two LFT's different from the identity commute if and only if
each one of them maps the set of fixed points of the other one
onto itself;
\par
- two LFT's other than the identity and with a common fixed point
in the extended plane $\widehat{\C}$ commute if and only if they
have the same set of fixed points in $\widehat{\C}$ .
\par
Behan \cite{Beh} proved that any two commuting self-maps of $\D$,
other than the identity or a hyperbolic automorphism, must have
the same Denjoy-Wolff point. It should be noted that in the case
of intertwining, even if $\f$ and $\p$ have the same Denjoy-Wolff
point $p$, the Denjoy-Wolff point of $f$ can be different from
$p$, as our Example~\ref{ex-noncom-infder} shows (with $p=1$ being
the common Denjoy-Wolff point of $\f$ and $\p$ and $z=0$ being the
Denjoy-Wolff point of $f$).
\par
An example of a function $f$, not an LFT, that commutes with an
elliptic automorphism other than the identity is given by
Example~\ref{ex-non-autom} and a description of all such maps is
given in \cite[Proposition~1.2.26]{A}. The self-maps of $\D$
commuting with a parabolic disk automorphism $\f$ are also known;
see \cite[Proposition~1.2.27]{A} for the statement adapted to a
half-plane, and also \cite{CDP1}. In this case the commutant
actually admits rather complicated maps, certainly not LFT's. We
will show, however, that the elliptic and parabolic automorphisms
$\f$ are the only possible exceptions.
\par
Having proved the earlier main theorems, we now deduce the result
on the commutant. The reader should note that it is not a mere
corollary of our previous theorems on intertwining. Namely, one
has to consider separately the elliptic and non-elliptic cases and it turns out that the elliptic case does not follow directly from the
corresponding result for intertwining.
\par
%%%%%%%%%%%%%%%%%%%%%
%% Theorem commute.
%%%%%%%%%%%%%%%%%%%%%
\begin{theorem}
 \label{thm-commute}
Let $\f$ be a linear fractional self-map of $\D$ which does not
fall into any of the following categories: the identity, an
elliptic automorphism, or a parabolic automorphism. If $f$ is a
self-map of $\D$ different from a constant or the identity map,
then $f\circ\f=\f\circ f$ holds if and only if $f$ is also an LFT
and has the same set of fixed points in $\widehat{\C}$ as $\f$.
Moreover, $f$ and $\f$ have the same dynamic type.
\end{theorem}
%%%%%%%%%%%%%%%%%%%%%
\begin{proof}
Denote by $p$ the Denjoy-Wolff point of $\f$. By Behan's Theorem
\cite{Beh}, either $p$ is the Denjoy-Wolf point of $f$, or $f$ is
a hyperbolic automorphism and $p$ is a fixed point of $f$. Thus,
if $f$ is either parabolic or hyperbolic but not an automorphism,
we have that $0<f^\prime(p)\le 1$, hence $f$ is conformal at $p$.
If $f$ is a hyperbolic automorphism then $f^\prime$ is finite,
hence $f$ is again conformal. In either one of these cases, the
result follows from Theorem~\ref{thm-inter-rigid}.
\par
In the elliptic case, it suffices to show that $f^\prime(p)\neq 0$ for the Denjoy-Wolff point $p$ of $\f$ and the statement will again follow from Theorem~\ref{thm-inter-rigid}. To this end, without loss of generality we may assume that $p=0$. (Otherwise consider the conjugate application $F=\f_p\circ f\circ\f_p$ with the fixed point at the origin which also satisfies $F^\prime (0)=f^\prime(p)$.) We can, thus, start off with the assumption that $\f(z)=z/(cz+d)$. Note that $d\neq 0$ since $f\not\equiv const$. Also, it follows by either  Proposition~\ref{prop-lft-self} or Lemma~\ref{lem-lft-self} that $|d|\neq 1$. Now there are two possibilities.
\par\medskip
1) $c=0$. In this case, $f(z/d)=f(z)/d$ for all $z$ in $\D$. This is readily seen to imply that $f(z)=a_1 z$ for some non-zero constant $a_1$ (recalling that $f\not\equiv const$ and $|d|\neq 1$), for example, either by using the power series development or by using the following argument. Differentiation yields $f^\prime(z/d)=f^\prime(z)$; hence by induction
$f^\prime(z/d^n)=f^\prime(z)$ for all $z\in\D$ so, by the uniqueness principle for analytic functions it follows that $f^\prime\equiv const$.
\par\medskip
2) $c\neq 0$. Suppose that, on the contrary, $f^\prime(0)=0$. Then there exists $N\geq 2$ and $a_N\neq 0$ such that
\[
 f(z)=\sum_{n=N}^\infty a_n z^n = a_N z^N + o(z^N) \,.
\]
In this context, as is usual, by $o(z^N)$ we mean any function $g$ analytic in $\D$ for which $\lim_{z\to 0} g(z)/z^N=0$. By a standard geometric series formula,
\[
\f(z)=\frac{z}{cz+d}=\frac{z}{d} + o(z)\,.
\]
Thus, on the one hand
\[
 f(\f(z)) = a_N \f(z)^N + o(\f(z)^N) = \frac{a_N}{d^N} z^N + o(z^N)
\]
and, on the other hand,
\[
 \f(f(z)) = \frac{f(z)}{d} + o(f(z)) = \frac{a_N}{d} z^N + o(z^N)\,.
\]
By equating the $N$-th coefficients in the last two formulas, it follows that $d=d^N$ which is impossible in view of the assumptions $d\neq 0$, $|d|\neq 1$ and $N\ge 2$. This completes the proof.
\end{proof}
%%%%%%%%%%%%%%%%%%%%%%%
\par
Obviously, in the extreme case when $f$ is a constant map, it
commutes with $\f$ if and only if $f$ is identically equal to a
fixed point of $\f$ but since we only admit those maps $f$ for
which $f(\D)\subset\D$, this can only happen in the case when $\f$
is elliptic.
\par
The fact that the commutation equation $f\circ\f=\f\circ f$ implies that $f$ is an LFT whenever $\f$ is such, discarding the exceptional cases mentioned earlier, was also obtained in \cite[Theorem~2]{JRS}
by using completely different methods.
%%%%%%%%%%%%%%%%%%%%%%%
\subsection{Roots}
\label{subs-proofs-roots}
%%%%%%%%%%%%%%%%%%%%%%%
The result about the commutant has an immediate application to the
questions of existence of the $n$-th ``roots'' of a given linear
fractional map, that is, of those self-maps $f$ of $\D$ such that
$f_n=\f$. It is trivial but important to observe that the $n$-th
iterate of an automorphism is trivially an automorphism. Also, an
arbitrary iterate of a self-map of $\D$ always has the same type
as the map itself.
\par
A disk automorphism of any type (other than the identity) is
easily seen to have an $n$-th root for arbitrary $n$ and this root
must also be an automorphism. Indeed, a trivial argument with
bijections shows that if $f$ maps $\D$ into itself, $f_n=\f$, and
$\f$ is a bijection of the disk, then $f$ must also be a bijective
map of $\D$.
\par
The identity map is again exceptional for it can be shown to have
uncountably many $n$-th roots for any $n$. One also easily checks
that, given any point $p\in\D$, if $\la$ denotes any $n$-th root
of unity different from $1$ and $\f_p$ is defined as in
\eqref{inv-autom}, then every disk automorphism of the form
$\f(z)=\f_p(\la \f_p(z))$ is an $n$-th root of $id_\D$ and $p$ is
its fixed point, hence $\f$ is elliptic. It is easy to check that
each pair $(p,\la)$ defines a different map, so there are indeed
uncountably many of them. It is not difficult to see that all
$n$-th roots of the identity map have to be elliptic automorphisms
by transferring the problem to the right half-plane if necessary.
\par
For other linear fractional self-maps $\f$ of $\D$ it turns out
that a solution $f$ to the equation $f_n=\f$ exists, it must be an
LFT and have the same type as $\f$. That is, another rigidity
principle holds here. The results are as follows.
\par
%%%%%%%%%%%%%%%%%%%%%
%% Theorem root-hp.
%%%%%%%%%%%%%%%%%%%%%
\begin{theorem}
 \label{thm-root-hp}
For any parabolic or hyperbolic linear fractional map $\f$, the
equation $f_n=\f$ has a unique solution. This solution $f$ is
parabolic whenever $\f$ is parabolic and hyperbolic whenever $\f$
is such.
\par
In either case, the formula for the solution is obtained as
follows: if $\tau$ denotes the Denjoy-Wolff point of $\f$, the map
$\widehat{\f}=T_\tau\circ \f \circ T_\tau^{-1}$ becomes simply
$\widehat{\f}(w)=A w + B$, a self-map of $\H$ with $A\ge 1$ and
Re\,$B\ge 0$. Denote by $\a$ the only positive $n$-th root of $A$.
Then
\[
 f(z) = T_\tau^{-1}(g(T_\tau(z)) = \tau \frac{g \( \frac{\tau+z}{\tau-z}
 \)-1}{g \( \frac{\tau+z}{\tau-z} \)+1} \,,
\]
where
\[
 g(w) = \a w + \frac{B}{\sum_{k=0}^{n-1} \a^k} \,.
\]
\end{theorem}
%%%%%%%%%%%%%%%%%%%%%
%%%%%%%%%%%%%%%%%%%%%
\begin{proof}
By virtue of the equation $f_n=\f$, it is immediate that $f$ and
$\f$ commute:
\[
 f\circ\f = f_{n+1} = \f\circ f \,.
\]
Hence, whenever $\f$ is not a parabolic automorphism,
Theorem~\ref{thm-commute} shows that $f$ is also an LFT. If $\f$
is a parabolic automorphism, each one of its roots is again a disk
automorphism, as observed earlier.
\par
Since in all possible cases both $\f$ and $f$ are linear
fractional transformations, the equation $f_n=\f$ becomes much
easier to solve in the right half-plane $\H$. After applying the
Cayley transform $T_\tau$, the equation becomes
$g_n=\widehat{\f}$, where the functions $\widehat{\f}=T_\tau\circ
\f \circ T_\tau^{-1}$ and $g= T_\tau \circ f \circ T_\tau^{-1}$
are self-maps of $\H$. As observed earlier, $\widehat{\f}$ can be
written in the form $\widehat{\f}(w)=A w+B$, with Re\,$B \geq 0$
always and with $A=1$ in the parabolic case and $A>1$ in the
hyperbolic case. Since $f$ is also an LFT with the same
Denjoy-Wolff point as $\f$, the function $g$ will have a similar
representation: $g(w)=a w+b$, $a\ge 1$ and Re\,$b\ge 0$.
\par
These representations allow us to find readily the formula for the
$n$-th iterate $g_n$. After identifying the coefficients of
$\widehat{\f}$ with those of $g_n$ in the equation
$g_n=\widehat{\f}$, we get the system of equations
\[
 a^n=A\,, \qquad b (1+a+a^2+\ldots+a^{n-1})=B \,.
\]
Writing $\a$ for the unique positive $n$-th root of $A$, we see
that
\[
 g(w) = \a w + \frac{B}{\sum_{k=0}^{n-1} \a^k}
\]
is one solution of the equation $g_n=\widehat{\f}$. It is actually
the only one. Indeed, with any other $n$-th root of $A$ instead of
the positive root $\a$, the map $g$ would no longer be a self-map
of the right half plane.
\par
Note also that the maps $\widehat{\f}$ and $g$ are of the same
type, since $\a$ and $A$ are simultaneously equal to, or greater
than, one.
\end{proof}
%%%%%%%%%%%%%%%%%%%%%
\par\smallskip
The situation is radically different for elliptic maps. It can be
seen that, for example, the function $\f(z)=-z/(3z+4)$ has no
square roots that map $\D$ into itself, while the automorphism
$\f(z)=-z$ has exactly two: $f_1(z)=i z$ and $f_2(z)=-i z$. The
problem is that, even though one can formally solve an equation
for obtaining a linear fractional root of a given LFT, this root
need not map the disk into itself. However,
Proposition~\ref{prop-lft-self} or Lemma~\ref{lem-lft-self} will
allow us to control effectively the number of solutions of the
equation $f_n=\f$ in the case of an elliptic map different from an
automorphism.
\par
%%%%%%%%%%%%%%%%%%%%%
%% Theorem root-ell.
%%%%%%%%%%%%%%%%%%%%%
\begin{theorem}
 \label{thm-root-ell}
If $\f$ is an elliptic linear fractional self-map of $\,\D$ other
than the identity and with Denjoy-Wolff point $p$, then the map
$\widehat{\f}=\f_p\circ\f \circ\f_p$ has the form
$\widehat{\f}(z)=\frac{Az}{Cz+1}$ with $|A|+|C|\leq 1$ and $A\neq
1$. The number of solutions of the equation $f_n=\f$ equals the
cardinality of the set
\[
 \{ a\in\C\,\colon\, a^n=A \text{ and } |C(1-a)|\le (1-|a|)\cdot
 |1-A|\,\} \,.
\]
Furthermore, each solution of the equation $f_n=\f$ is determined
by one such root $\,a$ and is again an elliptic LFT of the form
$f=\f_p\circ g\circ\f_p$, where
\[
 g(z)=\frac{a z}{\frac{C(1-a)}{1-A} z+1} \,.
\]
In particular, if $\f$ is an elliptic automorphism other than the
identity, there are exactly $n$ different solutions of the
equation $f_n=\f$.
\end{theorem}
%%%%%%%%%%%%%%%%%%%%%
\par
%%%%%%%%%%%%%%%%%%%%%
\begin{proof}
Let $\f$ be an elliptic linear fractional self-map of $\D$ with
Denjoy-Wolff point $p$ in $\D$. Denoting by $\f_p$ the
automorphism given by \eqref{inv-autom} which is an involution and
interchanges $p$ and the origin, it is immediate that the map
$\widehat{\f}=\f_p\circ\f\circ\f_p$ is an elliptic LFT with the
Denjoy-Wolff point $z=0$. The map $\widehat{\f}$ in this case can
be written as $\widehat{\f}(z)=\frac{Az}{Cz+1}$ and by applying
either Proposition~\ref{prop-lft-self} or
Lemma~\ref{lem-lft-self}, or working directly with appropriate
inequalities, we readily see that it is a self-map of $\D$ if and
only if $|A|+|C|\leq 1$.
\par
Assuming that $\f$ is not an automorphism and using
Theorem~\ref{thm-commute} again, we can show as before that every
possible solution of the equation $g_n=\widehat{\f}$ must have the
form $g(z)=\frac{a\,z}{c\,z+1}$. By matching the coefficients in
the equation $g_n=\widehat{\f}$ as before, we get the system of
equations
\begin{equation}
 a^n=A\,, \qquad c (1+a+a^2+\ldots+a^{n-1})=C \,.
 \label{syst}
\end{equation}
Note that for any $n$-th root $a$ of $A$, the sum
$\sum_{j=0}^{n-1} a^j\neq 0$. Otherwise we would have
$0=1-a^n=1-A$ and $C=0$, meaning that $\widehat{\f}$ and therefore
also $\f$ is the identity map, the case excluded from the start.
Thus, we can solve the system of equations \eqref{syst} and infer
that the functions of the form
\[
 g(z)=\frac{a z}{ {\frac {C} {\sum_{j=0}^{n-1}a^j}}z+1}
 =\frac{a z}{\frac{C(1-a)}{1-A} z + 1} \,,
\]
where $a$ denotes any one of the $n$-th roots of $A$, are the only
possible solutions of the equation $g_n=\widehat{\f}$. Reasoning
as before, such a map $g$ will be a self-map of $\D$ if and only
if
\begin{equation}
 |C (1-a)|\le (1-|a|)\cdot |1-A| \,.
 \label{root-cond}
\end{equation}
Therefore, the solutions of the equation $g_n=\widehat{\f}$, if
any, will be only those maps $g$ as above for which inequality
\eqref{root-cond} is verified. This easily leads to the formula
for the solutions of the equation $f_n=\f$ given in the statement
of the theorem.
\par
Now let $\f$ be an elliptic automorphism. The map
$\widehat{\f}=\f_p\circ \f\circ\f_p$ is an automorphism that fixes
the origin, hence a rotation: $\widehat{\f}(z)=\la z$,
$\la\in\partial\D\setminus\{1\}$. The map $f$ (and therefore also
$g$) is a disk automorphism other than the identity. Since $g$ and
$\widehat{\f}$ commute and $\widehat{\f}$ fixes the origin, the
same must be true of $g$, hence it is also a rotation. Thus, if
$\mu$ is any $n$-th root of $\la$, it is immediate that the
rotation $g(z)=\mu z$ is a solution of the equation
$g_n=\widehat{\f}$ and there can be no other solutions.
\end{proof}
%%%%%%%%%%%%%%%%%%%%%

%%%%%%%%%%%%%%%%%%%%%
\subsection{A remark on the Koenigs embedding problem for semigroups}
 \label{subs-proof-sgr}
%%%%%%%%%%%%%%%%%%%%%
\par\smallskip
We finally address some aspects of the so-called \textit{Koenigs
embedding problem\/} for semigroups of analytic functions. Recall
that an indexed family $\cg=\{g_t\,\colon\,t\in [0,\infty)\}$ of
analytic self-maps of $\D$ is said to be a \textit{semigroup} if
it is closed and additive under composition: $g_s\circ
g_t=g_{s+t}$ for all $s$, $t\in [0,\infty)$ and the function
$t\mapsto g_t$ is \textit{strongly continuous}, meaning that
$g_s\to g_t$ uniformly on the compact sets in $\D$ as $s\to t$.
Note that the additivity condition implies that $g_0$ is the
identity mapping. Much about semi-groups of analytic functions can
be found in the survey article \cite{Si} and in the monograph
\cite{Sho}.
\par
In relation to the elliptic case, it is clear that there are
semigroups of automorphisms in which the identity equals $g_t$ for
more than one value of $t$ (for example, when $g_{t/2}$ coincides
with some automorphism $\f_p$, an involution).
\par
The following is a natural and fundamental question about
semi-groups of analytic self-maps of the disk: when does such a
semi-group contain a linear fractional map? The answer has
recently been given in \cite[Theorem~3.2]{BCD}: this happens if and only if each member of the semi-group $\cg=\{g_t\,\colon\,t\in
[0,\infty)\}$ is a linear fractional map. Moreover, all of its members $g_t$ with $t\neq 0$ are linear fractional maps of the
same type as the given LFT. This can be seen as follows from our
findings here. Suppose that $\cg$ contains an LFT (not the
identity), say $\f=g_t$, for some $t>0$. As long as $\f$ is not an
elliptic or parabolic automorphism, the commutation relation
\[
 g_s\circ g_t=g_{s+t}=g_t\circ g_s
\]
and Theorem~\ref{thm-commute} will imply that any other member
$g_s$ of $\cg$ must also be an LFT and of the same type as $\f$.
If $\f$ is an automorphism, the conclusion follows in a different
but simple manner and is well known.
\par
The corollary below provides further details about embedding an
LFT into a semigroup. The first part of the statement (on the
uniqueness of the semigroup that contains a given LFT) must be
known to the experts in various cases and may have simpler proofs
as well. We state it here because it does not seem easy to give an
explicit reference that would cover all the cases, and also
because of the novelty of the method of proof. The main idea here
is to recover the unique semigroup from just one of its elements
by extracting successive square roots, thus relying on our
Theorem~\ref{thm-root-hp} and Theorem~\ref{thm-root-ell}. The same
idea appears in the second part, which refers to the elliptic
case. Our method also yields a criterion somewhat different from
the analytic condition presented in \cite{JRS}, \cite{KRS}, and
\cite[Proposition~5.9.5]{Sho}, as well as from the geometric criterion
given most recently in \cite[Proposition~3.4]{BCD}.
\par
%%%%%%%%%%%%%%%%%%%%%
%% Theorem semigr.
%%%%%%%%%%%%%%%%%%%%%
\begin{corollary}
 \label{cor-semigr}
The following assertions hold for semigroups of self-maps of the
unit disk.
\begin{enumerate}
 \item[(i)] Given a parabolic or hyperbolic linear fractional
self-map $\f$ of $\D$ and a positive number $s$, there is a unique
semigroup $\cg=\{g_t\,\colon\,t\in [0,\infty)\}$ of self-maps of
$\,\D$ such that $g_s=\f$. Every elliptic disk automorphism is a
member of some semigroup (not necessarily unique).
\par
\item[(ii)] Every elliptic, non-automorphic, linear fractional
self-map $\f$ of $\,\D$ is conjugate to a map of the form $(Az)/(C
z+1)$, where $|A|+|C|\le 1$, as was pointed out in
Theorem~\ref{thm-root-ell}. Such a map $\f$ belongs to some
semigroup of self-maps of the disk if and only if there exists a
sequence $\{a_n\}_{n=1}^\infty$ such that $a_0=A$, every $a_n$
($n\ge 1$) is one of the square roots of the number $a_{n-1}$, and
all $a_n$ belong to the non-tangential approach (Stolz type)
region
\[
 S=\left\{ z\in\D\,\colon\,\frac{|C|}{|1-A|} |1-z|\le 1-|z|
 \right\}
\]
with vertex at $z=1$.
\end{enumerate}
\end{corollary}
%%%%%%%%%%%%%%%%%%%%%
%%%%%%%%%%%%%%%%%%%%%
\begin{proof}
In order to prove assertion (i) of the theorem, let us assume that
$\f$ is either hyperbolic or parabolic. The function $\f$ will
always have a square root which is an LFT of the same type as
$\f$. In view of the rule $g_{t/2} \circ g_{t/2}=g_t=\f$, this
root must be equal to $g_{t/2}$. Since the root is again a map of
the same type as $\f$, it will have a unique root itself, and this
root is again an LFT of the same type as $\f$; we may proceed
inductively to conclude that every member of $\cg$ of the form
$g_{t/2^n}$ is determined uniquely by $\f$. By composition, $g_{m
t/2^n}$ is also determined in a unique fashion for all
non-negative integers $m$, $n$. The dyadic rational numbers
$m/2^n$ are dense in $[0,\infty)$, so by strong continuity every
member of the semigroup $\cg$ is determined uniquely by the
initial map $\f$.
\par
The statement for elliptic automorphisms can be proved in an
analogous way, except for the uniqueness part: $\f$ may have more
than one root but exactly one of these roots must equal $g_{t/2}$,
so pick that one and proceed as before to show that all maps $g_{m
t/2^n}$ are elliptic automorphisms. Finally, a standard arguments
involving Hurwitz's theorem shows that any locally uniform limit
of such maps is also univalent or a constant map. Since a
semigroup cannot contain constant maps, it follows that the limit
function is univalent; it follows rather easily that it must
actually be an elliptic automorphism.
\par
Let us now prove (ii). If an elliptic non-automorphic map $\f$
belongs to a semigroup $\cg$, say $\f=g_t$ with $t>0$, then
$g_{t/2}$ must be one of the possible square roots of $\f$ as
described in Theorem~\ref{thm-root-ell}. But we have seen that the
only possible square roots of the elliptic non-automorphic
function $(A z)/(C z+1)$ have the form
\[
 \frac{a z}{\frac{C (1-a)}{1-A}z+1} \,, \qquad a^2=A \,,
\]
where $|C (1-a)|\le (1-|a|) |1-A|$. An easy computation shows that
every possible square root of this root must have the form
\[
 \frac{b z}{\frac{C (1-b)}{1-A}z+1} \,, \qquad b^2=a \,,
\]
and is a self-map of $\D$ if and only if $|C (1-b)|\le (1-|b|)
|1-A|$. Thus, both the form of the map and the condition for being
a self-map are completely analogous to the previous inequality,
the number $a$ now being replaced by its square root $b$. This
allows us to conclude inductively that the sequence
$\{a_n\}_{n=1}^\infty$ defined by $a_0=A$ and $a_n$ being equal to
one of the square roots of the number $a_{n-1}$, $n\ge 1$, has the
required property: all numbers $a_n$ belong to the region
\[
 S=\left\{ z\in\D\,\colon\,\frac{|C|}{|1-A|} |1-z|\le 1-|z| \right\}
 \,.
\]
It should be pointed out that when $C\neq 0$ we can never get the
identity map in this process of root extractions. In the case
$C=0$, the region $S$ degenerates into the entire unit disk, so
the condition of belonging to it is trivially fulfilled and the
roots of $A z$ are easily found to be $a z$, where $a^n=A$.
\par
Conversely, if there exists a sequence $\{a_n\}_{n=1}^\infty$ as
above, its every member $a_n$ will satisfy condition
\eqref{root-cond} for the existence of the root of order $2^n$ of
the initial map $\f$, which allows us to proceed in the same way
as in the other cases and determine the unique semigroup $\cg$ to
which $\f$ belongs by determining all members $g_{t/2^n}$ first.
\end{proof}
%%%%%%%%%%%%%%%%%%%%%
\par
Note that in the above condition the sequence
$\{(1-|a_n|)/|1-a_n|\}_{n=0}^\infty$ is decreasing by the
definition of $a_n$ and the elementary inequality
\[
 \frac{1-|a_n|}{|1-a_n|} =
 \frac{1-|a_{n+1}|^2}{|1-a_{n+1}^2|} \ge
 \frac{1-|a_{n+1}|}{|1-a_{n+1}|}\,.
\]

\bigskip
%%%%%%%%%%%%%%%%%%%%%%%%%%%%%%%%%%%%%%%%%%%%%%%%%%%%%%%%%%%%%%%%%%%%%%


\begin{thebibliography}{99}
%%%%%%%%%%%%%%%%%%%%%%%%%%%%%%%%%%%%%%%%%%%%%%%%%%%%%%%%%%%%%%%%%%%%%%

\bibitem{A}
M. Abate, \textit{Iteration Theory of Holomorphic Maps on Taut
Manifolds\/}, Mediterranean Press, Commenda di Rende 1989.

\bibitem{Bea}
A. Beardon, \textit{The Geometry of Discrete Groups\/}, Graduate
Texts in Mathematics \textbf{91}, Springer-Verlag, New York -
Heidelberg - Berlin 1983.

\bibitem{Beh}
D. F. Behan, Commuting analytic functions without fixed points,
\textit{Proc. Amer. Math. Soc.} \textbf{37} (1973), no. 1,
114--120.

\bibitem{BG}
C. Bisi and G. Gentili, Commuting holomorphic maps and linear
fractional models, \textit{Complex Variables Theory Appl.}
\textbf{45} (2001), no. 1, 47--71.

\bibitem{BFHS}
P. S. Bourdon, E. E. Fry, C. Hammond, and C. H. Spofford, Norms of
linear-fractional composition operators, \textit{Trans. Amer.
Math. Soc.\/} \textbf{356} (2004), no. 6, 2459--2480.

\bibitem{BK}
D. Burns and S. G. Krantz, Rigidity of holomorphic mappings and a
new Schwarz lemma at the boundary, \textit{J. Amer. Math. Soc.\/}
\textbf{7} (1994),  no. 3, 661--676.

\bibitem{Br}
F. Bracci, Fixed points of commuting holomorphic mappings other
than the Wolff point, \textit{Trans. Amer. Math. Soc.\/}
\textbf{355} (2003), no. 6, 2569--2584.

\bibitem{BCD}
F. Bracci, M. D. Contreras, and S. D\'{\i}az-Madrigal,
Infinitesimal generators associated with semigroups of linear
fractional maps, \textit{J. Anal. Math.\/} \textbf{102} (2007),
119--142.

\bibitem{BTV}
F. Bracci, R. Tauraso, and F. Vlacci, Identity principles for
commuting holomorphic self-maps of the unit disc, \textit{J. Math.
Anal. Appl.\/} \textbf{270} (2002), no. 2, 451--473.

\bibitem{CDP1}
M. D. Contreras, S. D\'{\i}az-Madrigal, and Ch. Pommerenke, Some
remarks on the Abel equation in the unit disk, \textit{J. London
Math. Soc.\/} \textbf{75} (2007), 623-634.

\bibitem{CDP2}
M. D. Contreras, S. D\'{\i}az-Madrigal, and Ch. Pommerenke,
Iteration in the unit disk: the parabolic zoo. To appear in:
\textit{Proceedings of the Conference ``Complex and Harmonic
Analysis'' held in Thessaloniki, Greece in May of 2006}, Destech
Publications Inc. Editors: A. Carbery, P. L. Duren, D. Khavinson,
and A. G. Siskakis.

\bibitem{C1}
C. Cowen,  Iteration and the solution of functional equations for
functions analytic in the unit disk, \textit{Trans. Amer. Math.
Soc.} \textbf{265} (1981), no. 1, 69--95.

\bibitem{C2}
C. Cowen, Commuting analytic functions, \textit{Trans. Amer. Math.
Soc.} \textbf{283} (1984), no. 2, 685--695.

\bibitem{CM}
C. Cowen and B. MacCluer, \textit{Composition Operators on Spaces
of Analytic Functions\/}, Studies in Advanced Mathematics, CRC
Press, Boca Raton 1995.

\bibitem{JRS} F. Jacobzon, S. Reich, and D. Shoikhet, Linear
fractional mappings: invariant sets, semigroups, and commutativity, \textit{J. Fixed Point Theory Appl.} \textbf{5} (2009), no. 1, 63--91.

\bibitem{KRS}
V. Khatskevich, S. Reich, and D. Shoikhet, Abel-Schr\"oder
equations for linear fractional mappings and the Koenigs embedding
problem, \textit{Acta Sci. Math. (Szeged)} \textbf{69} (2003), no.
1-2, 67--98.

\bibitem{M}
M.~J.~Mart\'{\i}n, Composition operators with linear fractional
symbols and their adjoints, \emph{First Advanced Course in
Operator Theory and Complex Analysis (Seville 2004)\/}, pp.
105--112 (A. Montes-Rodr\'{\i}guez, editor), University of Seville
2006.

\bibitem{PC}
P. Poggi-Corradini, On the uniqueness of classical
semiconjugations for self-maps of the disk, \textit{Comput.
Methods Funct. Theory\/} \textbf{6} (2006), no. 2, 403--421.

\bibitem{P1} Ch. Pommerenke, Polymorphic functions for groups
of divergence type, \textit{Math. Ann.} \textbf{258} (1982),
353--366.

\bibitem{P2}
Ch. Pommerenke, \emph{Boundary Behaviour of Conformal Maps\/},
Springer-Verlag, Berlin 1992.

\bibitem{R}
W. Rudin, \emph{Real and Complex Analysis\/}, Third Edition,
McGraw-Hill, New York, St. Louis, etc. 1987.

\bibitem{Sha}
J. H. Shapiro, \textit{Composition Operators and Classical
Function Theory\/}, Springer-Verlag, New York 1993.

\bibitem{SSS}
J. H. Shapiro, W. Smith, and D. Stegenga, Geometric models and
compactness of composition operators, \textit{J. Funct. Anal.\/}
\textbf{127} (1995), no. 1, 21--62.

\bibitem{Shi}
A. L. Shields,  On fixed points of commuting analytic functions,
\textit{Proc. Amer. Math. Soc.} \textbf{15} (1964), 703--706.

\bibitem{Sho}
D. Shoikhet, \textit{Semigroups in geometrical function theory},
Kluwer Academic Publishers, Dordrecht 2001.

\bibitem{Si}
A. G. Siskakis, Semigroups of composition operators on spaces of
analytic functions, a review. In \textit{Studies on composition
operators (Laramie, WY, 1996)}, 229--252, \textit{Contemp. Math.}
\textbf{213}, Amer. Math. Soc., Providence, RI 1998.

\bibitem{Va}
G. Valiron, Sur l'it\'eration des fonctions holomorphes dans un
demi-plan, \textit{Bulletin Sc. math.\/} \textbf{55} (1931),
105--128.

\bibitem{Vl}
F. Vlacci, On commuting holomorphic maps in the unit disc of $\C$,
\textit{Complex Variables Theory Appl.} \textbf{30} (1996), no. 4,
301--313.
\end{thebibliography}
\end{document}